\documentclass[12pt]{amsart}
\usepackage{amssymb}
\usepackage{relsize}
\usepackage[utf8]{inputenc}
\usepackage[french, english]{babel}
\usepackage{tikz}
\usepackage{tikz-cd}
\usetikzlibrary{matrix,arrows}
\usepackage[a4paper, total={6.5in, 8in}, centering]{geometry}
\usepackage{enumitem}
\usetikzlibrary{positioning}
\definecolor{processblue}{cmyk}{0.96,0,0,0}
\usepackage[colorlinks=true,hyperindex]{hyperref}
\hypersetup{
    colorlinks,
    linkcolor=black,
    citecolor=cyan,
    urlcolor=magenta
}

\newcommand{\RR}{\mathbb{R}}
\newcommand{\QQ}{\mathbb{Q}}
\newcommand{\ZZ}{\mathbb{Z}}
\newcommand{\NN}{\mathbb{N}}

\newcommand{\GG}{\mathbb{G}}
\newcommand{\AAA}{\mathbb{A}}

\newcommand{\FF}{\mathbb{F}}

\newcommand{\gf}{\mathfrak{g}}

\newcommand{\wf}{\mathfrak{w}}

\newcommand{\oh}{\mathcal{O}}

\newcommand{\gh}{\mathcal{G}}

\newcommand{\ah}{\mathcal{A}}
\newcommand{\ih}{\mathcal{I}}
\newcommand{\hh}{\mathcal{H}}
\newcommand{\uh}{\mathcal{U}}

\newcommand{\eh}{\mathcal{E}}

\newcommand{\sh}{\mathcal{S}}
\newcommand{\Zh}{\mathcal{Z}}
\newcommand{\wh}{\mathcal{W}}
\newcommand{\ch}{\mathcal{C}}
\newcommand{\Nh}{\mathcal{N}}

\newcommand{\Res}{\text{R}}
\newcommand{\Rese}{\emph{R}}

\newcommand{\spec}{\text{Spec}\,}

\newcommand{\dimn}{\text{dim}\,}

\theoremstyle{plain}
\newtheorem{thm}{Théorème}[section]
\newtheorem{lem}[thm]{Lemme}
\newtheorem{propn}[thm]{Proposition}
\newtheorem{cor}[thm]{Corollaire}

\theoremstyle{definition}
\newtheorem{defn}[thm]{Définition}

\newcommand{\rpot}[1]{ (\hspace{-0,7mm}( {#1} )\hspace{-0,7mm}) }

\begin{document}
\title{Théorie de Bruhat-Tits pour les groupes quasi-réductifs}
\author[J. Lourenço]{João Lourenço}
\address{Mathematisches Institut der Universität Bonn, Endenicher Allee 60, 53115 Bonn, Deutschland}
\email{lourenco@math.uni-bonn.de}

\selectlanguage{french}
\begin{abstract}
Soient $K$ un corps discrètement valué et hensélien, $\oh$ son anneau d'entiers supposé excellent, $\kappa$ son corps résiduel supposé parfait et $G$ un $K$-groupe quasi-réductif, c'est-à-dire lisse, affine, connexe et à radical unipotent déployé trivial. On construit l'immeuble de Bruhat-Tits $\ih(G, K)$ pour $G(K)$ de façon canonique, améliorant les constructions moins canoniques de M. Solleveld sur les corps locaux, et l'on associe  un $\oh$-modèle en groupes $\gh_{\Omega}$ de $G$ à chaque partie non vide et bornée $\Omega$ contenue dans un appartement de $\ih(G,K)$. On montre que les groupes parahoriques $\gh_{\textbf{f}}$ attachés aux facettes peuvent être caractérisés en fonction de la géométrie de leurs grassmanniennes affines, ainsi que dans la thèse de T. Richarz. Ces résultats sont appliqués ailleurs à l'étude des grassmanniennes affines tordues entières. 
\end{abstract}
\vspace*{-2cm}
\maketitle
\tableofcontents

\section{Introduction}

La théorie des immeubles de F. Bruhat et J. Tits, cf. \cite{BTI} et \cite{BTII}, pour les groupes réductifs sur un corps valué est un outil surtout combinatoire, mais aussi schématique, qui a connu un succès énorme, tant en théorie des représentations qu'en géométrie algébrique arithmétique. Le but principal de ce travail est d'étendre toute cette théorie aux groupes quasi-réductifs sur les corps discrètement valués, excellents et résiduellement parfaits, dont des indications fortes sont déjà parus depuis longtemps dans le travail \cite{Oes} de J. Oesterlé sur les nombres de Tamagawa, où des diverses propriétés arithmétiques des groupes unipotents ployés ont été remarquées pour la première fois, et aussi dans l'excellent livre \cite{BLR} de S. Bosch, W. Lütkebohmert et M. Raynaud, dans lequel le modèle de Néron (localement de type fini) des groupes quasi-réductifs nilpotents fut construit, cf. 10.2 de \cite{BLR}. Plus récemment, la classification des groupes pseudo-réductifs a permis de se battre avec tels phénomènes dans des cadres assez généraux, comme exemplifié par : la démonstration \cite{ConFin} par B. Conrad de la finitude des nombres de Tamagawa pour les groupes algébriques linéaires ; la construction par M. Solleveld en \cite{Sol} des immeubles de Bruhat-Tits pour les groupes quasi-réductifs sur les corps locaux, bien que des décompositions de Cartan et d'Iwasawa ; et la preuve \cite{Ros} de la dualité de Tate pour les groupes commutatifs pseudo-réductifs par Z. Rosengarten.

Décrivons les protagonistes de cet article, à savoir les groupes pseudo- et quasi-réductifs. Soient $K$ un corps commutatif et $G$ un $K$-groupe affine, lisse et connexe. Alors $G$ est dit pseudo-réductif (resp. quasi-réductif) s'il ne contient aucun sous-groupe distingué non trivial, qui soit lisse, unipotent et connexe (resp. et de plus déployé). A. Borel et J. Tits avaient déjà annoncé sans démonstration en \cite{BoT2}, que cette classe partage une grande quantité de propriétés en commun avec les groupes réductifs, comme par exemple, l'existence d'un système de racines relatif $\Phi(G,S)$, des sous-groupes radiciels associés aux racines, d'une bonne notion de sous-groupes pseudo-paraboliques, et des théorèmes de conjugaison de ces derniers bien que des tores déployés maximaux.  

Néanmoins, il s'agit d'une découverte récente de B. Conrad, O. Gabber et G. Prasad \cite{CGP}, complétée à la caractéristique $p=2$ lorsque $[K:K^2]>2$ en \cite{CP}, qu'il est loisible de classifier les groupes pseudo-réductifs de façon élégante. Comme cela est toujours remplit par les corps ultramétriques résiduellement parfaits et à complété séparable, on supposera dorénavant qu'on a toujours $[K:K^p]\leq p$. Alors le théorème principal de \cite{CGP} affirme que tout groupe pseudo-réductif $G$ s'obtient à partir de certains groupes pseudo-réductifs primitifs, en prenant des restrictions des scalaires, des produits et en modifiant le Cartan. Plus précisément, soient $K'$ une $K$-algèbre finie réduite, $G'$ un $K'$-groupe à fibres pseudo-réductives primitives, $C'$ un Cartan fixe de $G'$ et $C$ un groupe pseudo-réductif commutatif opérant sur $\Res_{K'/K}G'$ de telle sorte que $\Res_{K'/K} C'$ soit fixé ; alors, le quadruple $(K'/K, G', C', C)$ est dénommé une présentation de $G=(C \ltimes \Res_{K'/K}G')/\Res_{K'/K}C'$ et tout groupe pseudo-réductif admet une telle présentation, qui ne dépend que du choix d'un Cartan $C$ de $G$. Il ne reste qu'à éclaircir le sens du mot primitif, qu'on applique pour décrire les trois classes suivantes : les groupes réductifs, semi-simples, simplement connexes et absolument presque simples ; les groupes exotiques basiques, qui n'existent qu'en caractéristique $p=2$ ou $3$ et dont le diagramme de Dynkin du système de racines irréductible possède une arête étiquetée avec $p$ ; les groupes absolument non réduits basiques notés $\text{BC}_n$, qui sont les seuls groupes pseudo-déployés dont le système de racines est non-réduit de type $BC_n$. On signale que ces constructions en petites caractéristiques étaient déjà parues dans l'œuvre de J. Tits, cf. \cite{TitsOber} et \cite{TitsBour}, lequel a été motivé par des réalisations entières de certains groupes associés aux algèbres de Kac-Moody affines tordues, voir aussi \cite{TitsUniq} et \cite{TitsTwin}.

Soit maintenant $(K, \omega)$ un corps discrètement valué, hensélien, excellent et résiduellement parfait. Étant donné un groupe quasi-réductif $G$ sur $K$, le but présent est de construire des données radicielles valuées, des immeubles et des schémas en groupes parahoriques. Il faut dire que la partie combinatoire, n'impliquant que les points rationnels de $G$, a déjà été traitée par M. Solleveld, en exploitant crucialement le fait que, par exemple pour un groupe exotique basique $G$, les points rationnels $G(K)$ coïncident avec ceux de son cousin déployé $\overline{G}$, ou que le groupe des points rationnels de $\text{BC}_n$ s'identifie à $\text{Sp}_{2n}(K^{1/2})$. Une telle approche devient toutefois problématique lorsqu'on veut traiter de questions schématiques (et même la définition d'une donnée radicielle valuée est subtile, en vertu de ce que l'application $G \rightarrow \overline{G}$ induit une bijection non additive entre leurs systèmes de racines). Par conséquent, on partira du zéro en définissant ces structures combinatoires, essayant de traiter les facteurs primitifs de $G$ aussi égalitairement que possible.

\begin{thm}
Soient $K$ un corps discrètement valué, hensélien, $\oh$ son anneau d'entiers supposé excellent et $\kappa$ son corps résiduel supposé parfait. Soient $G$ un groupe quasi-réductif sur $K$, $S$ un tore maximal déployé de $G$, $Z:=Z_G(S)$ le centralisateur de $S$ et $U_a$ le sous-groupe radiciel attaché à $a \in \Phi:=\Phi(G,S)$. L'ensemble des valuations $\varphi=(\varphi_a)$ de la donnée radicielle abstraite $(Z(K), U_a(K))_{a \in \Phi}$ formées de celles qui sont compatibles à $\omega$, c'est-à-dire telles que $\varphi_a(zuz^{-1})-\varphi_a(u)=\omega(a(z))$, pour tout $z \in Z(K)$, $a\in \Phi$ et $u \in U_a(K)$, porte la structure d'un espace affine non vide dont l'espace vectoriel réel sous-jacent $V$ s'identifie au sous-espace de $X_*(S)\otimes \RR$ engendré par $\Phi^{\vee}$.
\end{thm}

Nous notons que la preuve de ce théorème se fait en deux étapes : premièrement, on suppose que $G$ est quasi-déployé et l'on fait usage de certains quasi-systèmes de Chevalley de $G^{\text{psréd}}$ ; puis on utilise toute la théorie de Bruhat-Tits développée dans le cas quasi-réductif (voir le théorème ci-dessous) pour descendre les valuations le long de l'extension non ramifiée $K \rightarrow K^{\text{nr}}$. L'espace affine de ce théorème est l'appartement $\ah(G, S, K)$ associé au tore maximal déployé $S$ de $G$, il admet une structure de complexe simplicial engendrée par les hyperplans $\partial\alpha_{a,k}:=\{v \in V: a(v) +k=0\}$, $k \in \Gamma_a^{\varphi}$, qui ne dépendent pas du choix d'une origine $\varphi$, sur lequel les points rationnels $N(K)$ du normalisateur $N$ de $S$ opèrent. Pour définir le immeuble de Bruhat-Tits $\ih(G, K)$, il suffit de recoller tous les appartements $\ah(G, gSg^{-1}, K)$, en requérant que chaque $\psi=\varphi+v$ appartenant à $\ah(G, S, K)$ soit fixé par le sous-groupe parahorique $P_{\psi}$ engendré par $\Zh(\oh)$ et les $U_{a,-a(v)}:=\varphi_a^{-1}([-a(v), +\infty])$, où l'on note $\Zh$ l'unique $\oh$-modèle en groupes parahorique de $Z$.

Le prochain objectif est la construction des modèles en groupes entiers $\gh_{\Omega}$ de $G$ associés à une partie non vide et bornée $\Omega$ contenue dans un appartement de $\ih(G, K)$, dont la procédure consistera à définir auparavant les modèles correspondants $\Zh$ de $Z$, resp. $\uh_{a, \Omega}$ de $U_a$ pour toute racine non divisible $a \in \Phi_{\text{nd}}$, et ensuite à les recoller en $\gh_{\Omega}$ à l'aide de lois birationnelles. Observons qu'on aura besoin comme précédemment de traiter en priorité le cas quasi-déployé, pour en déduire le général par descente non ramifiée. En tout cas, les théorèmes de structure peuvent être énoncés comme suit :

\begin{thm}
Pour chaque partie non vide et bornée $\Omega$ d'un appartement de $\ih(G, K)$, il existe un plus grand modèle en groupes $\gh_{\Omega}$ lisse, affine et connexe de $G$ tel que $\gh_{\Omega}\otimes K=G$ et $\gh_{\Omega}(\oh^{\emph{nr}})$ fixe $\Omega$. Si $\Omega=\emph{\textbf{f}}$ est une facette, alors ses sous-facettes correspondent bijectivement aux $k$-sous-groupes paraboliques de $\gh_{\emph{\textbf{f}},s}:=\gh_{\emph{\textbf{f}}} \otimes k$ par l'application $\emph{\textbf{f'}} \succcurlyeq \emph{\textbf{f}} \mapsto Q_{\emph{\textbf{f'}}, \emph{\textbf{f}}}:=\emph{im}(\gh_{\emph{\textbf{f'}},s}\rightarrow \gh_{\emph{\textbf{f}},s})$. De plus, on a $P_{\emph{\textbf{f'}}}=Q_{\emph{\textbf{f'}}, \emph{\textbf{f}}}(k)\times_{\gh_{\emph{\textbf{f}},s}(k)}P_{\emph{\textbf{f}}} $.
\end{thm}

Enfin on décrit le lien entre les modèles parahoriques $\gh$ et la projectivité d'un certain ind-schéma parfait, appelé la grassmannienne affine $\text{Gr}_{\gh}$ de $\gh$, qui paramètre les $\gh$-torseurs sur $\spec W_{\oh}(R)$ trivialisés au-dessus de $\spec W_{\oh}(R)\otimes K$. Il a été justement pour analyser ces espaces (voir l'article \cite{LouGr} qui étendre quelques résultats de \cite{PR} et/ou \cite{PZh}), que nous fûmes conduits à développer la théorie de Bruhat-Tits pour les groupes quasi-réductifs.

\begin{thm}
Soient $K$ un corps discrètement valué, complet et résiduellement parfait, $\oh$ son anneau d'entiers et $\gh$ un $\oh$-groupe lisse et affine. Pour que la grassmannienne affine $\emph{Gr}_{\gh}$ de $\gh$ soit ind-projective, il faut et il suffit que $\gh^0$ soit un modèle parahorique de sa fibre quasi-réductive.
\end{thm}

On en tire aussitôt un corollaire concernant la projectivité générique de la grassmannienne affine globale de A. Beilinson et V. G. Drinfeld \cite{BDr}, qui permet de géométriser une vieille conjecture de \cite{BLR} sur les modèles de Néron globales.

\subsection*{Leitfaden}
Le \ref{section theorie de groupes} rappelle au lecteur les notions fondamentales de la théorie des groupes pseudo-réductifs et quasi-réductifs ainsi que certains concepts nouveaux concernant ceux-ci dont on aura besoin dans la suite. Ensuite le \ref{section combinatoire} introduit les données radicielles valuées, les appartements et les immeubles. Enfin, dans \ref{section geometrique} nous construisons les schémas en groupes attachés aux parties bornées contenues dans un appartement de l'immeuble, étudions leurs radicaux unipotents ainsi que le lien entre les groupes parahoriques et les sous-groupes paraboliques de la fibre spéciale. Le \ref{section grassmanniennes} donne un critère géométrique de parahoricité en termes des grassmanniennes affines.

\subsection*{Remerciements} Le présent travail est sans aucun doute immensément redevable à F. Bruhat et J. Tits, qui avaient déjà développé la machinerie nécessaire pour étudier les groupes quasi-réductifs sur un corps discrètement valué et hensélien. J'ai profité aussi de ma participation à l'école d'été \og Immeubles et grassmanniennes affines \fg{}, durant laquelle j'ai écrit et présenté une partie de mes travaux. D'ailleurs j'ai entrepris des conversations intéressantes sur ce thème avec B. Conrad, A. Ivanov, G. Prasad, T. Richarz, P. Scholze et M. Solleveld. Je suis très reconnaissant à M. Santos, P. Scholze et au rapporteur de leurs suggestions pour améliorer ce manuscrit. Cet article a été rédigé dans le cadre et avec le soutien du Sonderforschungsbereich/Transregio 45, ainsi que du Leibniz-Preis de la Deutsche Forschungsgemeinschaft attribué au Prof. Dr. Peter Scholze.

\subsection*{Notations} Sauf mentionné autrement, tous les corps à partir de \ref{donnees radicielles abstraites et valuees, appartements} seront supposés discrètement valués, henséliens, excellents (à savoir, $\hat{K}/K$ est une extension séparable) et résiduellement parfaits. On note $\Res_{B/A}$ le foncteur de restriction des scalaires le long d'une extension finie et plate d'anneaux noethériens $A \rightarrow B$. Les immeubles considérés seront presque toujours réduits et notés $\ih(G,K)$ ; leurs variantes élargies seront notées $\ih^1(G,K)$.

\section{Rappels sur la théorie des groupes}\label{section theorie de groupes}
Dans la suite on va rappeler et énoncer quelques propriétés qui concernent les groupes quasi-réductifs, en s'appuyant sur les travaux de B. Conrad, O. Gabber et G. Prasad, cf. \cite{CGP}, \cite{CP}.

\subsection{} \label{proprietes generales des groupes pseudo-reductifs}Commençons par rappeler au lecteur la notion du $K$-radical unipotent $R_{u, K} G$ d'un $K$-groupe affine, lisse et connexe $G$: c'est le plus grand sous-groupe lisse, connexe, unipotent et distingué de $G$ défini sur $K$. Le groupe $G$ est dit pseudo-réductif si $R_{u,K}G$ est trivial, dont  l'exemple prototypique est la restriction des scalaires à la Weil $G=\Res_{K'/K} G'$ d'un groupe réductif $G'$ le long d'une extension de corps $K'/K$ purement inséparable, $G$ n'étant réductif que lorsque $K'=K$ ; nous allons expliquer leur classification au prochain paragraphe. De manière analogue, on définit le radical unipotent déployé $R_{ud, K}G$ en exigeant de plus que le sous-groupe unipotent soit déployé, c'est-à-dire il possède une suite croissante de sous-groupes unipotents et connexes à pièces graduées consécutives isomorphes à $\GG_a$, ce qui ne se produit pas forcément au-dessus des corps imparfaits. Les groupes à radical unipotent déployé trivial sont appelés quasi-réductifs, cf. \cite{BTII}, 1.1.11., dont les unipotents parmi eux s'appellent ployés. On va tout d'abord éclaircir les théorèmes structuraux valides pour les groupes pseudo-réductifs et quasi-réductifs, dont la presque totalité avait été déjà annoncée par A. Borel et J. Tits en \cite{BoT2}.

Soient $K$ un corps et $G$ un $K$-groupe quasi-réductif. Les $K$-tores déployés maximaux de $G$ forment une seule classe de conjugaison sous $G(K)$ d'après le th. C.2.3 de \cite{CGP} et considérons un tel sous-tore $S$. Le système de racines $\Phi:=\Phi(G,S)$ est l'ensemble des poids de $S$ dans l'algèbre de Lie $\text{Lie} \,G$ pour la représentation adjointe, il porte la structure d'un système de racines dans le $\QQ$-espace vectoriel de $X_*(S) \otimes \QQ$ engendré par les racines et peut être même promu à une donnée radicielle $(X_*(S), \Phi^{\vee}, X^*(S), \Phi)$ (cf. th. C.2.15 de \cite{CGP}). D'autre part, on peut considérer aussi le système de racines absolu $\tilde{\Phi}=\Phi(G_{K^{\text{sép}}}, T_{K^{\text{sép}}})$, où $T$ désigne un tore maximal de $G$ contenant $S$, sur lequel le groupe galoisien $\Gamma:=\text{Gal}(K^{\text{sép}}/K)$ opère, de sorte que les racines $\Phi$ correspondent bijectivement aux $\Gamma$-orbites dans $\tilde{\Phi}$ dont la restriction à $S$ ne s'annule pas. Quant aux $K$-sous-groupes pseudo-paraboliques $P$ de $G$, ils peuvent être définis de façon dynamique, pour laquelle nous renvoyons à la déf. 2.2.1 de \cite{CGP}, ou de façon combinatoire à conjugaison près comme les seuls sous-groupes connexes et lisses contenant $S$ dont l'algèbre de Lie est la somme des espaces propres $\gf_0$ et $\gf_a$, où $a$ parcourt une partie parabolique de $\Phi$, cf. lem. C.2.16 de \cite{CGP}. Selon la prop. B.4.4 de \cite{CGP}, la seule opération permise d'un tore sur un groupe unipotent ployé est l'identité. Par suite, le radical unipotent $R_{u,K}G$ d'un groupe quasi-réductif $G$ est contenu dans $Z:=Z_G(S)$, d'où une identification canonique $\Phi(G, S)=\Phi(G^{\text{psréd}}, S^{\text{psréd}})$ et des isomorphismes $U_a \cong U^{\text{psréd}}_a$, où l'on note $G^{\text{psréd}}$ le quotient pseudo-réductif de $G$ etc. pour les restants groupes. On remarque que le centralisateur $Z$ de $S$ est quasi-réductif et qu'il en est de même des Cartans de $G$, lesquels sont de plus nilpotents (et même commutatifs dans le cas pseudo-réductif, voir la prop. 1.2.4 de \cite{CGP}).

\subsection{} \label{classification des groupes pseudo-reductifs}

On peut construire encore d'autres groupes pseudo-réductifs en modifiant un Cartan: soient $K'$ une $K$-algèbre finie et réduite, $G'$ un $K'$-schéma en groupes à fibres pseudo-réductives (c'est-à-dire, lorsqu'on écrit $K'=\prod_{i \in I} K_i'$ comme produit de corps, alors $G'$ s'écrit comme produit $\prod_{i \in I } G_i' $ avec $G_i'$ un groupe pseudo-réductif sur $K_i'$), $C'\subseteq G'$ un Cartan, $C$ un $K$-groupe pseudo-réductif commutatif muni d'une opération à gauche sur $\Res_{K'/K} G'$ qui laisse $\Res_{K'/K}C'$ stable point par point et d'une application $\Res_{K'/K}C' \rightarrow C$ équivariante pour l'action sur $\Res_{K'/K} G'$; alors le $K$-groupe lisse, affine et connexe $G:=(\Res_{K'/K} G' \rtimes C)/\Res_{K'/K}C'$ obtenu en plongeant $\Res_{K'/K} C$ antidiagonalement dans le produit semi-direct, est pseudo-réductif, cf. prop. 1.4.3 de \cite{CGP}. Nous disons que le quadruple $(K'/K, G', C', C)$ est une présentation du couple $(G, C)$ (ou parfois, par abus de langage, juste de $G$, lorsque le Cartan $C$ est sous-entendu). Le théorème de classification des groupes pseudo-réductifs affirme l'existence des présentations \og primitives \fg{} dans presque tous les cas :

\begin{thm}[B. Conrad, O. Gabber et G. Prasad]\label{presentation de groupes pseudo-reductifs}
Soient $K$ un corps tel que $[K:K^2]\leq 2$ lorsque sa caractéristique est égale à $2$, $G$ un groupe pseudo-réductif sur $K$ et $C$ un Cartan de $G$. Alors il y a une et une seule présentation $(K'/K, G', C', C)$ de $(G,C)$ à isomorphisme unique près, dont les fibres de $G'$ sont primitives. De plus, la $K$-algèbre $K'$ et le $K'$-groupe $G'$ ne dépendent que de $G$ à isomorphisme unique près. 
\end{thm}

Expliquons le sens du mot \og primitif \fg{} appliqué dans l'énoncé. Cela va dénommer une classe de certains groupes qui sont minimaux par rapport aux présentations ci-dessus, et qui sont divisés en trois volets : le volet réductif, formé des groupes semi-simples, simplement connexes et absolument presque simples (et ceci est le seul volet non vide lorsque $p\geq 5$) ; l'exotique, formé des groupes pseudo-réductifs pseudo-simples exotiques basiques, qui n'existent qu'en caractéristique $p=2$ ou $3$, lorsque le diagramme de Dynkin possède une $p$-arête ; le non réduit, formé des groupes pseudo-réductifs, pseudo-simples et pseudo-déployés à système de racines non réduit, qui n'existent qu'en caractéristique $p=2$. Notons encore que l'hypothèse $[K:K^2]\leq 2$ dès que $p=2$ qu'on a faite ne peut même pas être soulevée, car sinon les groupes pseudo-réductifs deviennent typiquement n'importe quoi, voir \cite{CP}. Ceci nous conduit à faire dorénavant l'hypothèse suivante, à laquelle les corps valués considérés dans la suite satisferont automatiquement : on a $[K:K^p] \leq p$.

Décrivons alors les bizarreries propres aux caractéristiques petites, en commençant par ce qu'on appelle les groupes pseudo-réductifs exotiques basiques (pseudo-déployés), dont la construction repose sur l'existence de diagrammes de Dynkin possédant une arête étiquetée avec $p$, alors de type $B_n$, $C_n$, $n \geq 2$ ou $F_4$ si $p=2$ ou $G_2$ si $p=3$. Alors, un groupe pseudo-réductif $G$ sur un corps imparfait $K$ en caractéristique $p=2$ ou $3$ est dénommé exotique basique si le radical unipotent $R_{u,K^{\text{alg}}} G$ est défini sur $K^{1/p} \neq K$, le quotient réductif $G':=G^{\text{réd}}_{K^{1/p}}$ défini au-dessus de $K^{1/p}$ est semi-simple, absolument presque simple et simplement connexe, dont le diagramme de Dynkin admet une arête avec étiquette $p$, et l'application correspondante $G \rightarrow \Res_{K^{1/p}/K} G'$ est une immersion fermée identifiant $G$ à l'image réciproque d'un Levi $\bar{G}$ de $\Res_{K^{1/p}/K}\bar{G}'$ par la restriction des scalaires de l'isogénie très spéciale $\Res_{K^{1/p}/K} \pi: \Res_{K^{1/p}/K}G'\rightarrow \Res_{K^{1/p}/K}\bar{G}'$ (voir la construction de \cite{CGP}, 7.1). Réciproquement, pour qu'une telle donnée $(G', \bar{G})$, formée d'un $K^{1/p}$-groupe réductif $G'$ et d'un $K$-groupe réductif $\bar{G}$ soumis aux conditions antérieures, produise un groupe pseudo-réductif en prenant l'image réciproque $G=(\Res_{K^{1/p}/K}\pi)^{-1}(\bar{G})$, il faut et il suffit que cette image soit lisse ou que $\bar{G}$ soit contenu dans l'image de $\Res_{K^{1/p}/K}\pi$, ce qui est typiquement faux, cf. prop. 7.3.1 de \cite{CGP}.

Exemplifions notamment la construction des groupes pseudo-réductifs exotiques basiques dans le cas pseudo-déployé, qui jouera un rôle plus important dans la construction des immeubles. Supposons que $G':=H$ est un $\FF_p$-groupe épinglé, presque simple et simplement connexe, ayant une $p$-arête dans son diagramme de Dynkin. Alors on peut regarder l'isogénie très spéciale $\pi$ comme étant définie au-dessus de $\FF_p$ et il s'avère que $\bar{G}':=\bar{H}$ est un $\FF_p$-groupe épinglé, presque simple et simplement connexe à diagramme de Dynkin dual, de telle sorte que l'on ait $T_{\bar{H}}=\pi(T_H)$, $N_{\bar{H}}=\pi(N_H)$, $W_H \xrightarrow{\simeq}W_{\bar{H}}$ et que l'endomorphisme $\GG_a \xrightarrow{x_a} U_{a, H} \xrightarrow{\pi} U_{\bar{a}, \bar{H}} \xrightarrow{x_{\bar{a}}^{-1}} \GG_a$ soit de la forme $t \mapsto t$ (resp. $t \mapsto t^p$) lorsque $a$ est longue (resp. courte), où $\bar{\Phi}$ s'identifie au dual de $\Phi$ par la réciproque de la bijection $a \in \Phi \mapsto \bar{a} \in \bar{\Phi}$. Il est aisé de voir que l'image réciproque $G$ du Levi $\bar{H}$ de $\Res_{K^{1/p}/K} \bar{H}$ par $\Res_{K^{1/p}/K}\pi$ satisfait aux conditions ci-dessus, donc il s'agit d'un groupe pseudo-réductif exotique basique pseudo-déployé à système de racines $\Phi$, cf. th. 7.2.3 de \cite{CGP}. Observons que la grosse cellule de $G$ peut être donnée par les formules suivantes : $Z=\prod_{a \in \Delta} \Res_{K_a/K} a^{\vee}(\GG_m)$ et $U_a=\Res_{K_a/K} \GG_a$ dans le groupe pseudo-épinglé $\Res_{K^{1/p}/K} H$, où $K_a=K$ si $a$ est longue (resp. $K_a=K^{1/p}$ sinon).

Si $p=2$, il existe un autre phénomène largement inusité, celui des groupes pseudo-réductifs dont le système de racines absolu n'est pas réduit. Pour les construire, les auteurs de \cite{CGP} considèrent le groupe $G'=\text{Sp}_{2n}$ canoniquement épinglé à système de racines de type $C_n$ et commencent par observer que $\Phi:=\Phi' \cup (\frac{1}{2}\Phi' \cap X^*(T')) $ est un système de racines non réduit de type $BC_n$. À chaque racine ni divisible ni multipliable $a$ (resp. multipliable $c$, resp. divisible $2c$) dans $\Phi$, on lui associe le sous-groupe radiciel $U_a=\Res_{K^{1/2}/K}\GG_a$ (resp. $U_c=\Res_{K^{1/2}/K}\GG_a \times \GG_a$ muni de la loi de groupe additive naturelle, resp. $U_{2c}= 1 \times \GG_a \subseteq U_{c}$) ; on munit ensuite ces sous-groupes des applications vers les $\Res_{K^{1/2}/K}U_a'$ (resp. $\Res_{K^{1/2}/K} U_{2c}'$) définies par les formules $x \mapsto x$ (resp. $(x,y)\mapsto \alpha x^2+y$, où $\alpha \in K^{1/2}\setminus K$). \footnote{Signalons que $U_c\rightarrow \Res_{K^{1/2}/K} U_{2c}'$ ne dépend pas du choix de $\alpha$, puisque, lorsqu'on se donne $\beta=\alpha r^2+c $ avec $r \in K^{1/2}$, $c\in K$, alors on obtient un et un seul isomorphisme $U_{c,\beta} \rightarrow U_{c,\alpha}$, $(x,y) \mapsto (rx,y+cx^2)$ qui préserve les applications vers $\Res_{K^{1/2}/K} U_{2c}'$.} Or le th. 9.6.14. de \cite{CGP} fournit une loi de groupe sur le $K$-schéma $U^+=\prod_{a \in \Phi^+}U_a$, qui est indépendante de l'ordre mis sur les facteurs du produit à un sens convenable et aussi compatible avec la projection sur $\Res_{K^{1/2}/K}U'^{+}$ ; de plus, si l'on note $Z$ le groupe pseudo-réductif commutatif $\Res_{K^{1/2}/K}T'$, alors il opère sur $U^+$ et le schéma  $U^+ \times w \times ZU^+$ admet une et une seule loi birationnelle stricte recouvrant celle de la cellule ouverte de Bruhat de $\Res_{K^{1/2}/K} G'$. Le $K$-groupe $G:=\text{BC}_n$ échéant est lisse, affine, connexe et même pseudo-réductif muni d'un morphisme canonique $\pi: G \rightarrow \Res_{K^{1/2}/K} G'$, dont le noyau est contenu dans le produit des sous-groupes radiciels multipliables. De plus, son quotient réductif géométrique est défini sur $K^{1/2}$ et s'identifie à $G'=\text{Sp}_{2n}$ par l'application déduite de $\pi$ par adjonction. Ce groupe pseudo-réductif ne dépend que de $n$, ses $K$-automorphismes sont tous intérieures et il ne possède pas des $K$-formes non triviales. Les groupes isomorphes à $\text{BC}_n$ sont dénommés absolument non réduits basiques dans \cite{CGP} (ne faire pas confusion avec le sens du mot \og réduit \fg{} en géométrie algébrique !)

Rappelons un petit peu de la structure générale des groupes pseudo-réductifs qui fut développée dans \cite{CP}. Une extension centrale de groupes pseudo-réductifs dérivés est dite modérée si le noyau (qui est souvent non lisse) ne contient aucun sous-groupe unipotent et on peut montrer (voir th. 5.1.4 de \cite{CP}) qu'il existe une telle extension centrale modérée universelle $\widetilde{G}$ pour chaque groupe $G=G^{\text{dér}}$, laquelle s'identifie canoniquement au groupe $\Res_{K'/K}G'$ du th. \ref{presentation de groupes pseudo-reductifs}. En général, on notera encore $\widetilde{G}$ le revêtement modérément universel du groupe dérivé $G^{\text{dér}}$, qui doit être regardé comme l'équivalent pseudo-réductif du revêtement simplement connexe d'un groupe réductif. Finalement, d'après le th. 2.4.1 de \cite{CGP}, on dispose d'un groupe affine de type fini $\text{Aut}_{G,C}$ classifiant les automorphismes de $G$ dont la restriction à un Cartan donné $C$ est triviale et l'on note $C^{\text{aut}}$ son sous-groupe lisse maximal - le modifié $G^{\text{aut}}:=(G\rtimes C^{\text{aut}})/C$ jouera souvent un rôle semblable à celui du groupe adjoint pour les groupes réductifs (voir la suite et le th. C.2.15 de \cite{CP}).

\subsection{} \label{quasi-deploye}

Rappelons qu'un groupe quasi-réductif $G$ sur un corps $K$ s'appelle quasi-déployé si l'un de (équiv. tous) ses sous-groupes pseudo-boréliens est défini sur $K$. Ceci revient au même que demander que le centralisateur $Z$ d'un tore déployé maximal soit nilpotent.

\begin{propn}
Soient $K$ un corps de caractéristique $p$ tel que $[K:K^p]\leq p$, $G$ un $K$-groupe quasi-réductif et $\widetilde{G}=\Rese_{K'/K}G'$ le revêtement modérément universel du quotient pseudo-réductif $G^{\emph{psréd}}$, où les fibres de $G'$ sont pseudo-réductives primitives. Pour que $G$ soit quasi-déployé il faut et il suffit que les fibres réductives (resp. exotiques) de $G'$ soient quasi-déployées (pseudo-déployées).
\end{propn}

\begin{proof}
Observons que $G$ est quasi-déployé si et seulement si son quotient pseudo-réductif $G^{\text{psréd}}$ l'est, car la commutativité de $Z^{\text{psréd}}$ est nécessaire et suffisante pour que $Z$ soit résoluble. On voit aussitôt que, lorsque $G$ est pseudo-réductif, $Z$ est commutatif si et seulement si $\widetilde{Z}$ l'est. Maintenant si $G=\widetilde{G}=\Res_{K'/K}G'$ est modérément universel et $G'$ primitif, on voit que l'application, qui à un pseudo-parabolique $P'$ de $G'$ fait correspondre $P=\Res_{K'/K} P'$, définit une bijection entre les ensembles de sous-groupes pseudo-paraboliques de $G$ et de $G'$, qui préserve les sous-groupes pseudo-boréliens parmi eux.

Finalement il faut voir que, si $K'=K$ et $G=G'$ est exotique basique et quasi-déployé, alors $G$ est pseudo-déployé. Pour cela, on considère l'épimorphisme $f: G' \rightarrow \bar{G}$ vers le cousin déployé, lequel à son tour est quasi-déployé, en observant que $f(C')=f(Z_{G'}(S'))=Z_{\bar{G}'}(f(S'))$ est un Cartan. Mais voilà que les diagrammes de Dynkin non simplement lacés ne possèdent aucun automorphisme extérieur, d'où le fait que $\bar{G}$ soit déployé et même que $G$ soit pseudo-déployé, en vertu du cor. 7.3.4 de \cite{CGP}.
\end{proof}

\begin{cor}\label{pseudo-steinberg}
Gardons l'hypothèse faite sur $K$ et supposons que sa dimension cohomologique est inférieure ou égale à $1$. Alors tout $K$-groupe quasi-réductif $G$ est quasi-déployé.
\end{cor}

\begin{proof}
Grâce à la proposition précédente, il suffit de le vérifier pour chaque groupe simplement connexe ou exotique basique $G$, auquel cas cela résulte du théorème de Steinberg appliqué à $G$ si celui-ci est réductif ou à $\bar{G}$ sinon.
\end{proof}

\subsection{} \label{quasi-epinglages}

Maintenant on va rappeler les notions de quasi-épinglages et quasi-systèmes de Chevalley pour les groupes réductifs quasi-déployés et puis on les généralisera au cas pseudo-réductif et quasi-déployé, sous l'hypothèse $[K:K^p]\leq p$ lorsque $p$ est positif. Alors, durant ce paragraphe, on fixera un corps $K$ de caractéristique $p\geq 0$ tel que $[K:K^p]\leq p$ lorsque $p>0$, un $K$-groupe pseudo-réductif et quasi-déployé $G$, un sous-tore déployé maximal $S$ de $G$ et une base $\Delta$ du système de racines $\Phi$ de $G$ suivant $S$, qui correspond à un pseudo-Borel $B$ de $G$. On note $T$ le plus grand sous-tore de $G$ contenant $S$ (unique, vu que $G$ est quasi-déployé), $\tilde{\Phi}$ le système de racines absolu $\tilde{\Phi}$ de $G$ suivant $T$ et $\tilde{\Delta}$ la base de $\tilde{\Phi}$ associée à $B$. 

Si $G$ est un $K$-groupe réductif déployé, alors un épinglage de $G$ correspond classiquement à la donnée d'isomorphismes $x_a: \GG_a \rightarrow U_a$ pour chaque racine simple $a \in \Delta$. Il est souvent utile d'étendre un épinglage en un système de Chevalley, cf. \cite{BTII} et \cite{SGA3}, c'est-à-dire la donnée d'isomorphismes $x_a: \GG_a \rightarrow U_a$ pour toute racine $a\in \Phi$ de sorte que $m_a=x_a(1)x_{-a}(1)x_a(1)$ normalise le tore déployé maximal donné $S$ et que les applications $\text{int}(m_a)\circ x_b$ et $x_{s_a(b)}$ de $\GG_a$ dans $U_{s_a(b)}$ sont égales au signe près. Observons qu'un tel prolongement est déterminé au signe près par l'épinglage choisi et son existence résulte de la prop. 6.2 de \cite{SGA3}. D'ailleurs, la première condition concernant $m_a$ peut s'exprimer en affirmant l'existence d'un  épimorphisme $\text{SL}_2 \rightarrow G^a$ tel que la restriction au sous-groupe canoniquement épinglé des matrices triangulaires supérieures (resp. inférieures) unipotentes coïncide avec $x_a$ (resp. $x_{-a}$ affecté du signe $-1$), où $G^a$ désigne le sous-groupe engendré par $U_{\pm a}$. 

Lorsque $G$ n'est plus déployé, mais bien quasi-déployé sur $K$, l'œuvre \cite{BTII} s'appuie sur un épinglage de Steinberg ou système de Chevalley-Steinberg, cf. 4.1.3. de \cite{BTII}, qui sont une sorte de donnée de descente des mêmes objets dans le cas déployé par rapport au groupe de Galois absolu, pour décrire les sous-groupes radiciels relatifs de $G$. Pour être précis, soit $\tilde{K}/K$ une extension galoisienne déployant $G$ et considérons un épinglage (resp. système de Chevalley) de $G\otimes \tilde{K}$ telle que $\gamma \circ x_{\tilde{a}} \circ \gamma^{-1}=x_{\gamma(\tilde{a})}$ (resp. la même condition affectée d'un signe $\pm 1$ lorsque $\tilde{a} \in \tilde{\Phi}$ s'écrit comme la somme de deux racines conjuguées). Étant donné un épinglage de Steinberg, on sait déjà comment le prolonger en un système de Chevalley, mais il faut encore vérifier la formule $\gamma \circ x_{\tilde{a}} \circ \gamma^{-1}=x_{\gamma(\tilde{a})}$ pour les racines $a \in \Phi$ dont la restriction à $S$ n'est pas divisible. Cela découle plus ou moins explicitement des travaux de R. Steinberg \cite{St1} et \cite{St2}, dont un traitement détaillé et moderne se trouve dans \cite{LndvCpc}, prop. 4.4. Armé d'un système de Chevalley-Steinberg, le reste du 4.1. de \cite{BTII} est consacré à en déduire des quasi-épinglages, (c'est-à-dire des isomorphismes avec certains groupes plus familiers) des $K$-sous-groupes radiciels $U_a$ de $G$ pour chaque racine non-divisible $a \in \Phi$. Par exemple, si $a$ n'est pas multipliable et si l'on note $K_{\tilde{a}}$ le corps fixe du stabilisateur dans $\Gamma=\text{Gal}(\tilde{K}/K)$ d'un relèvement $\tilde{a}$ de $a$, alors on obtient un épimorphisme central $\zeta_a: \Res_{K_a/K} \text{SL}_2 \rightarrow G^a$ donné par $ g \mapsto \prod_{\gamma(\tilde{a}) } \zeta_{\gamma(\tilde{a})}(\gamma(g))$ ; on note $x_a$ la restriction de $\zeta_a$ au sous-groupe des matrices triangulaires supérieures unipotentes identifié à $\Res_{K_a/K} \GG_a$ et dont l'image isomorphe n'est autre que le sous-groupe radiciel $U_a$. Dans le cas où $a$ est multipliable, alors l'épimorphisme central admet pour source le groupe $\Res_{K_{\widetilde{2a}}/K} \text{SU}_{3,K_{\tilde{a}}/K_{\widetilde{2a}}}$, où les relèvements de $a$ et $2a$ sont choisis de telle façon que leur différence soit une racine.

Éclaircissons aussitôt la structure du groupe unitaire spécial à trois variables $\text{SU}_{3, L/L_2}$ associé à l'extension quadratique séparable $L/L_2$, qui sera regardé comme le sous-groupe de $\Res_{L/L_2} \text{SL}_3$ qui préserve la forme hermitienne $e_1\sigma(e_3)+e_2\sigma(e_2)+e_3\sigma(e_1)$, en concordance avec \cite{BTII}, 4.1.9 ; le groupe unipotent des matrices triangulaires supérieures de $\text{SU}_{3,L/L_2}$ s'identifie alors à la sous-variété $H(L,L_2)$ de $\Res_{L/L_2} \AAA^2_L$ définie par $v+\sigma(v)=u\sigma(u)$, moyennant le morphisme suivant :
$$\mu(u,v)=  \begin{pmatrix} 
1 & -\sigma(u) & -v \\
0 & 1 & u \\
0 & 0 & 1 
\end{pmatrix}$$
qui admet aussi l'écriture $\mu(u,v)=x_{\tilde{a}}(u)x_{\tilde{a}+\sigma(\tilde{a})}(-v)x_{\sigma(\tilde{a})}(\sigma(u))$ pour un certain choix approprié d'épinglage de $\text{SL}_3$, cf. 4.1.9. de \cite{BTII}. Transportée à $H(L,L_2)$, la loi de groupe du sous-groupe radiciel de $\text{SU}_{3,L/L_2}$ s'écrit comme suit : 
 $$(u_1, v_1)\cdot (u_2, v_2)=(u_1+u_2, v_1+v_2+\sigma(u_1)u_2). $$
On observe que, fixant un élément $\lambda \in L$ de trace $1$, on peut identifier $H(L,L_2)$ par $j_{\lambda}: (u,v) \mapsto (u, v-\lambda u \sigma(u))$ au $L_2$-groupe $H^{\lambda}$ constitué par la sous-variété de $\Res_{L/L_2} \AAA^2 $ définie par l'équation $y+\sigma(y)=0$, où $(x,y) \in \Res_{L/L_2} \AAA^2_L$, et muni de la loi de multiplication : $$(x_1,y_1)\cdot (x_2,y_2)\mapsto (x_1+x_2, y_1+y_2-\lambda x_1\sigma(x_2)+\sigma(\lambda x_1)x_2 ).$$ 

Revenant au but initial de quasi-épingler les groupes pseudo-réductifs et quasi-déployés, on est confronté au fait que la dimension des sous-groupes radiciels est typiquement plus grande que $1$, même dans le cas pseudo-déployé, donc leur groupe d'automorphismes est en général trop gros et, par exemple, ne préserve pas la structure de droite sur une extension non triviale de $K$. Néanmoins on remarque que le groupe de $K$-automorphismes de $\Res_{L/K} \text{SL}_2$ fixant le Cartan est isomorphe à $L^{\times}$ par l'application qui fait correspondre à $l \in L^{\times}$ l'automorphisme 
$$ \begin{pmatrix} 
a & b \\
c & d 
\end{pmatrix} \mapsto \begin{pmatrix} 
a & lb \\
l^{-1}c & d 
\end{pmatrix},$$
tel qu'on le souhaite. Cela montre que le bon choix est de quasi-épingler les revêtements universels $\widetilde{G}^a$ des sous-groupes $G^a$ de rang pseudo-déployé $1$ :
 
\begin{defn}\label{definition quasi epinglage}
Un quasi-épinglage (resp. quasi-système de Chevalley) de $G$ suivant $S$ est la donnée d'isomorphismes $\zeta_a$ d'une restriction des scalaires de $\text{SL}_2$, $\text{SU}_3$ ou $\text{BC}_{1}$ sur $\widetilde{G}^a$ pour chaque racine simple $a \in \Delta$ (resp. toute racine non divisible $a \in \Phi_{\text{nd}}$), dont les restrictions au tore diagonal $\GG_m$ s'identifient aux coracines $a^{\vee}:\GG_m \rightarrow S$ et qui envoient le groupe des matrices triangulaires supérieures unipotentes sur $U_a$ (resp. et de sorte que $\text{int}(m_a) \circ \zeta_{ b}$ est semblable à $\zeta_{ s_a(b)}$ pour chaque couple $ a,  b \in \Phi_{\text{nd}}$).
\end{defn}

Ici l'on note $m_a$ l'image dans $G$ de la matrice anti-diagonale $\text{adiag}(-1, 1)$ (resp. $-\text{adiag}(1,1,1)$, resp. l'image réciproque de $m_a\in \text{SL}_2(L^{1/2})$) si $\widetilde{G}^a$ est identifiée par $\zeta_a$ à $\Res_{L/K} \text{SL}_2$ (resp. $\Res_{L_2/K} \text{SU}_{3,L/L_2}$, resp. $\Res_{L/K}\text{BC}_1$). Deux homomorphismes ayant pour source des restrictions des scalaires (pas nécessairement égales !) des groupes ci-dessus sont dits semblables s'ils se déduisent l'un de l'autre par conjugaison galoisienne et/ou par une involution de signe induite par un point rationnel de $Z^{\text{aut}}[2]$ (ceci n'est pas à craindre en caractéristique $2$, donc rendu superflu dans le cas de $\text{BC}_1$).

\begin{propn}\label{proposition quasi-epinglages}
Soient $K$ un corps de caractéristique $p \geq 0$ tel que $[K:K^p] \leq p$ lorsque $p>0$, $G$ un groupe pseudo-réductif et quasi-déployé, $S$ un sous-tore déployé maximal de $G$. Alors les classes de similitude des quasi-systèmes de Chevalley de $G$ suivant $S$ forment un espace principal homogène sous $Z^{\emph{aut}}(K)/Z^{\emph{aut}}[2](K)$ pour l'action de conjugaison.
\end{propn}

\begin{proof}
Observons qu'il suffit de vérifier que les classes modulo $\Gamma$ des quasi-épinglages forment un $Z^{\text{aut}}(K)$-torseur et que tout quasi-épinglage se prolonge en un quasi-système de Chevalley, qui est unique à similitude près. Soient $(\zeta_a)_{a\in \Delta}$ et $(\zeta_a')_{a \in \Delta}$ deux quasi-épinglages de $G$. Opérant par $\Gamma$, on peut supposer que $\zeta_a'\circ \zeta_a^{-1}$ est un automorphisme bien défini de $\tilde{G}^a$ qui fixe $\tilde{Z}^a$, donc induit par un élément de $\widetilde{Z}^{a,\text{aut}}$. Enfin il ne reste qu'à observer que $Z^{\text{aut}}=\prod_{a \in \Delta} \widetilde{Z}^{a,\text{aut}}$ d'après les lem. 6.2.2., éq. 6.2.3. et lem. C.2.3 de \cite{CP}. 

Quant à la deuxième affirmation, on choisit, pour chaque racine $a \in \Phi$, une écriture $a=s_{a_n}s_{a_{n-1}}\dots s_{a_1}(a_0)$, où les $a_i \in \Delta$ sont des racines simples positives, et l'on pose $\zeta_a:=\text{int}(m_{a_n}\dots m_{a_1}) \circ \zeta_{a_0}$. Notons que tout quasi-système de Chevalley étendant le quasi-épinglage donné prendre par définition même cette forme à similitude près. Il reste à constater la similitude de $\text{int}(m_a)\circ \zeta_b $ et $\zeta_{s_a(b)}$. Pour cela, on se ramène aisément au cas où $G=\widetilde{G}$ est primitif et même simplement connexe, en appliquant les morphismes $\pi$ de $ G$ dans $\overline{G}$ dans le cas exotique ou dans $\Res_{K^{1/2}/K}\text{Sp}_{2n}$ dans le cas absolument non réduit, puisqu'on n'a besoin que de comparer les restrictions $x_a$ des $\zeta_a$ aux sous-groupes radiciels $U_a$. Mais le cas simplement connexe résulte de l'existence de systèmes de Chevalley-Steinberg : en effet, si $\widetilde{G}^a \cong \Res_{L/K}\text{SL}_2$, alors $m_a=\prod_{\gamma \tilde{a}}m_{\gamma\tilde{a}}$ et $\text{int}(m_a)\circ x_b(u)=\prod_{\gamma\tilde{b}}x_{s_a(\gamma\tilde{b})}(\pm \gamma(u))$, où l'on note $s_a=\prod_{\gamma\tilde{a}}s_{\gamma\tilde{a}}$ le produit des réflexions de tous les relèvements de $a$, ce qui est évidemment semblable à $x_{s_a(b)}$ ; le cas de $\text{SU}_{2n+1}$ est laissé au soin du lecteur (voir les 4.1.9. et 4.1.11. de \cite{BTII}).
\end{proof}

\section{Valuations, appartements et immeubles}\label{section combinatoire}
Dans cette section, on se donne un corps discrètement valué, hensélien, excellent et résiduellement parfait $K$, un groupe quasi-réductif et quasi-déployé $G$ sur $K$ et un sous-tore déployé maximal $S$ de $G$. À ces données on leur associe des valuations au sens de la déf. 6.2.1. de \cite{BTI}, moyennant lesquelles les 6. et 7. de \cite{BTI} fournissent l'appartement $\ah(G,S,K)$ et l'immeuble $\ih(G, K)$ de Bruhat-Tits. Nous montrons aussi que les immeubles construits ici coïncident avec ceux de M. Solleveld, cf. \cite{Sol}.
\subsection{} \label{donnees radicielles abstraites et valuees, appartements}
Soit $(K, \omega)$ un corps discrètement valué, hensélien, tel que l'anneau d'entiers $\oh$ soit excellent, le corps résiduel $\kappa$ soit parfait et $\omega(K^{\times})=\ZZ$. Alors $K$ vérifie l'inégalité $[K:K^p]\leq p$ lorsqu'il est de caractéristique $p$, vu que le homomorphisme de complétion $K \rightarrow \widehat{K}$ est une extension séparable de corps et que $\widehat{K} \cong \kappa\rpot{\varpi}$ avec $\kappa$ parfait. Soient $G$ un $K$-groupe quasi-réductif quasi-déployé et $S$ un $K$-tore déployé maximal de $G$, de telle manière que son centralisateur $Z$ devienne un Cartan. Rappelons que $(Z(K), U_a(K))_{a \in \Phi}$ porte la structure d'une donnée radicielle (abstraite) génératrice dans $G(K)$ au sens du 6.1 de \cite{BTI}, d'après la rem. C.2.28 de \cite{CGP} (même dans le cas non quasi-déployé). Le reste du paragraphe sera consacré à améliorer cet énoncé, en munissant cette donnée de valuations au sens de la déf. 6.2.1. de \cite{BTI}. 

Considérons un quasi-système de Chevalley de $G^{\text{psréd}}$ dont l'existence est assuré par la prop. \ref{proposition quasi-epinglages} et reprenons les notations de \ref{quasi-epinglages}, où le but $U_a^{\text{psréd}}$ de $x_a$ sera identifié à $U_a$ au moyen de $q:G \rightarrow G^{\text{psréd}}$. On lui associe des fonctions $\varphi_a: U_a(K) \rightarrow \RR \cup \{+\infty\}$ inspirées du 4.2.2. de \cite{BTII} et données par les formules suivantes:

\begin{defn} 
Soit $a \in \Phi_{\text{nd}}$ une racine non divisible. On définit la fonction $\varphi_a: U_a(K) \rightarrow \RR \cup \{+\infty\}$ de la manière suivante :
\begin{itemize} 
\item si $\text{dom}(\zeta_a)=\Res_{K_a/K}\text{SL}_2$, alors $\varphi_a(x_a(r))=\omega(r)$ ; 
\item si $\text{dom}(\zeta_a)= \Res_{K_{2a}/K}\text{SU}_{3,K_a/K_{2a}}$, alors $\varphi_a(x_a(u, v)) =\frac{1}{2}\omega(v)$ ; 
\item si $\text{dom}(\zeta_a)= \Res_{K_{2a}/K}\text{BC}_1$, alors $\varphi_a(x_a(x,y))= \frac{1}{2}\omega(\alpha x^2+y)$.
\end{itemize}
Pour les racines divisibles, on définit $\varphi_{2a}$ comme étant la restriction de $2\varphi_a$ à $U_{2a}(K)$.
\end{defn}

Signalons que cette construction joue un rôle essentiel dans tout l'article, en étant la base sur laquelle toutes les remarques et preuves suivantes s'appuient. En fait, cette famille de fonctions induit une valuation des points rationnels, comme l'on montre ci-dessous:

\begin{propn}\label{verificacao da valorizacao}
La donnée $\varphi=(\varphi_a)_{a \in \Phi}$ pour toute racine $a \in \Phi$, est une valuation de la donnée radicielle abstraite $(Z(K), U_a(K))_{a \in \Phi}$ et elle est, en plus, compatible à $\omega$ au sens où $\varphi_a(zuz^{-1})-\varphi_a(u)=\omega(a(z))$ pour tout $z \in Z(K)$ et $u \in U_a(K)$.
\end{propn}

\begin{proof}
Observons d'abord qu'on peut supposer sans perdre de généralité que $G=G^{\text{psréd}}$ est pseudo-réductif. Alors, (V 0) et (V 4) sont évidemment remplies, tandis que (V 1) et (V 5) découlent des calculs classiques avec les groupes $\text{SL}_2$ ou $\text{SU}_3$, puisque $\text{BC}_1(K)=\text{SL}_2(K^{1/2})$. Les mêmes arguments avec les groupes classiques montrent que, pour déduire (V 2), il suffit d'en vérifier lorsque $m=m_a$ et de montrer la compatibilité de $\varphi$ avec $\omega$, c'est-à-dire que la différence $\varphi_a(zuz^{-1})-\varphi_a(u)$ est égale à $\omega(a(z))$ pour tout $z \in Z(K)$ et $u \in U_a(K)$. Ici on regarde $a$ comme un caractère rationnel de $Z$, c'est-à-dire, un élément du $\QQ$-espace vectoriel $X_K^*(Z)\otimes \QQ\cong X^*(S)\otimes\QQ$ - pour voir l'injectivité du morphisme de restriction, on n'a qu'à remarquer que $Z/S$ est une extension d'un groupe unipotent par un tore anisotrope ; la surjectivité résulte à son tour de la prop. VII.1.5. de \cite{Ray}. et l'on définit $\omega\circ a$ par la formule $\frac{1}{N}\omega\circ Na$ pour $N$ assez grand. Notons que l'opération de $Z$ sur $U_a$ se factorise en $Z \rightarrow Z^{a, \text{aut}}$ suivi de l'opération canonique de ce dernier sur $U_a$. Selon le th. 1.3.9 et la prop. 9.8.15 de \cite{CGP}, le groupe $Z^{a, \text{aut}}$ s'identifie à $\Res_{K_a/K} T^{\text{ad}}$ lorsque $\tilde{G}^a=\Res_{K_a/K}\text{SL}_2$ (resp. $\Res_{K_{2a}/K} T^{\text{ad}}$ lorsque $\tilde{G}^a=\Res_{K_{2a}/K}\text{SU}_{3,K_a/K_{2a}}$, resp. $\Res_{K_{2a}/K} Z$ lorsque $\tilde{G}^a=\Res_{K_{2a}/K}\text{BC}_1$). En particulier, la compatibilité résulte des calculs de \cite{BTII}, 4.2.7., pour $\text{SL}_2$ ou $\text{SU}_3$.

Finalement on vérifie (V 3), pour lequel on peut supposer $G$ primitif et même de rang pseudo-déployé égal à $2$. Si $G$ est réductif, l'inclusion $(U_{a,k}, U_{b,l})\subseteq \prod_{pa+qb \in \Phi, (p,q)\in \NN^2} U_{pa+qb, pk+ql}$ peut être déduite des formules de l'appendice de \cite{BTII}. Si $G$ est exotique, alors on voit par construction que la valuation est déduite d'une de $\Res_{K^{1/p}/K}G'$ par restriction, donc on se ramène au cas précédent. Supposons enfin que $G\cong \text{BC}_1$ est absolument non réduit et commençons par observer que, si les racines $a, b$ ne sont pas multipliables, l'axiome résulte du cas classique de $\text{Sp}_{4}$. Si les rayons radiciels attachés à $a$ et $b$ sont les deux pluriels, alors on constate sans peine que les sous-groupes radiciels commutent. Il reste à étudier le cas où $a$ est multipliable, $b$ n'est ni multipliable ni divisible et les racines $a+b$, $2a+b$ et $2a+2b$ appartiennent à $\Phi$. Dans ce cas, l'on a $(U_{a,k}, U_{b,l})=(\bar{U}_{2a,2k}, \bar{U}_{b,l})\subseteq \bar{U}_{2a+b, 2k+l}\bar{U}_{2a+2b, 2k+2l}=U_{a+b, k+l}U_{2a+b, 2k+l}$, où les $\bar{U}$ désignent les sous-groupes radiciels de $\text{Sp}_{4}$.
\end{proof}

On a construit une donnée radicielle valuée $(Z(K), U_a(K), \varphi)_{a \in \Phi}$ dans les points rationnels $G(K)$ du groupe quasi-réductif et quasi-déployé donné et maintenant on peut appliquer la théorie du 6.2 de \cite{BTI} pour obtenir l'appartement (réduit) de $G$ par rapport à $S$, comme on va expliquer dans la suite.

\begin{defn}
 L'appartement $\ah(G, S, K)$  est l'ensemble de toutes les valuations $\psi=(\psi_a)$ de la donnée radicielle abstraite considérée équipollentes à $\varphi=(\varphi_a)$, c'est-à-dire, telles que la fonction $\psi_a-\varphi_a$ soit constante égale à $a(v)$ pour un certain vecteur $v$ appartenant à $X_*(S^{\text{dér}}) \otimes \RR$ indépendant de $a$, où $S^{\text{dér}}:=(S \cap G^{\text{dér}})_{\text{réd}}^0$. Cet ensemble vient muni d'une structure de complexe polysimplicial engendrée par les hyperplans $\partial \alpha_{a,k}:= \{a(x-\varphi)+ k=0\}$ pour $k \in \Gamma_a'=\{ \varphi_a(u): u \in U_a(K)\setminus \{1\}, \varphi_a(u)=\text{sup }\varphi_a(uU_{2a}) \}$ et $a \in \Phi$.
\end{defn} 
 
En particulier, on voit que l'appartement ne dépend pas du choix de quasi-système de Chevalley de $G$, vu que leurs classes de similitude sont conjugués par $Z^{\text{aut}}(K)$ (voir le cor. ci-dessous ou bien le cor. 6.2.8 de \cite{BTI}, qui permet de déterminer le vecteur de translation par rapport à une base $\Delta$ de $\Phi$). Les racines affines de l'appartement sont précisément les demi-espaces $\alpha_{a,k}:=\{a(x-\varphi)+ k\geq 0\}$, $k \in \Gamma_a$, on note $\Sigma$ l'ensemble de celles-ci et on définit l'échelonnage $\eh \subseteq \Sigma \times \Phi$ comme étant la partie des $(\alpha,a)$ tels que $\alpha=\alpha_{a,k}$ avec $k \in \Gamma_a'$. Le groupe de Weyl affine $W_{\text{af}}$ est le groupe de transformations affines de $\ah(G, S, K)$ engendré par les réflexions $r_\alpha$ par rapport aux bords $\partial \alpha$ des racines affines. Le groupe $N(K)$ opère sur $\ah(G, S, K)$ par la règle suivante: si l'on pose $\psi= n \cdot \varphi $, alors  $\psi_a(u)=\varphi_{w^{-1}(a)}(n^{-1} u n)$. Cette opération est bien définie, préserve les racines affines et son image dans le groupe d'automorphismes de $\ah(G, S, K)$ contient $W_{\text{af}}$ (voir la prop. 6.2.10. de \cite{BTI} pour toutes ces affirmations). On sait déjà qu'un élément $z \in Z(K)$ opère sur $\ah(G, S, K)$ comme la translation de vecteur $\nu(z)$ tel que $a(\nu(z))=-\omega(a(z))$. En particulier, le plus grand sous-groupe borné $Z(K)_b$ de $Z(K)$, dont l'existence sera constatée dans la prop. \ref{modele de neron}, opère de manière triviale sur $\ah(G, S, K)$. On note aussi que l'opération de $N(K)$ détermine $\varphi_a(u)$, $u \in U_a(K)$ comme étant l'unique nombre réel $k$ tel que l'automorphisme de $\ah(G, S, K)$ induit par $m(u)$ soit la réflexion de bord $\partial \alpha_{a,k}$, où l'on note $m(u)$ le seul élément de $N(K)$ tel que $m(u)=u'uu''$ avec $u', u'' \in U_{-a}(K)$. Par suite, on tire le corollaire suivant, qui concerne les valuations compatibles à $\omega$.

\begin{cor}
Si $\psi$ est une valuation de $(Z(K), U_a(K))_{a \in \Phi} $ compatible à $\omega$, alors $\psi \in \ah(G,S, K)$ est équipollente à l'une des valuations de Chevalley-Steinberg construites ci-dessus.
\end{cor}

\begin{proof}
Quitte à faire une translation, on peut trouver $\varphi \in \ah(G,S,K)$ telle que $ \varphi(u_a)=\psi(u_a)$ pour un certain $u_a \in U_a(K)^*$, où $a$ parcourt une base $\Delta$ de $\Phi$. En particulier, si l'on identifie les appartements correspondants à $\varphi$ et $\psi$, en prenant ces points pour l'origine, il résulte par hypothèse que cette identification est aussi invariante par l'action de $N(K)$ par des automorphismes affines, d'où l'égalité $\varphi=\psi$ (voir aussi la rem. 6.2.12. b de \cite{BTI}).
\end{proof}

Parfois il s'avère plus utile de considérer l'appartement élargi $\ah^1(G, S, K)$, défini comme le produit $\ah(G, S, K)\times X_*(S^c)\otimes \RR$, où $S^c$ est le plus grand tore central de $S$ et $X_*(S^c)\otimes \RR$ est vu comme un complexe simplicial sans facettes propres et non vides. Afin de prolonger l'opération de $N(K)$, on note que $X_*(S^c)\otimes \RR$ s'identifie naturellement au dual de $X_K^*(G)\otimes \RR=X_K^*(N/N^{\text{dér}})\otimes \RR$ et on pose $n\cdot(x, v)=(nx, v+\theta(n)) $, où $\theta(n) \in X_*(S^c)\otimes {\RR}$ est l'unique élément tel que $\chi(\theta(n))=-\omega(\chi(n)) $ pour tout $\chi \in X_K^*(G)$. On remarque aussitôt que $z \in Z(K)$ opère par la translation de vecteur $\nu(z)+\theta(z)$ caractérisé par la formule $\chi(\nu(z)+\theta(z))=-\omega \circ \chi(z) $ pour tout caractère rationnel $\chi \in X_K^*(Z)\otimes \QQ$, ce qu'on vérifie aussitôt en faisant usage de l'orthogonalité des décompositions $X_*(S)\otimes \RR=X_*(S^c)\oplus X_*(S^{\text{dér}})$ resp. $X^*(S)\otimes \RR=X^*(S^c)\otimes \RR \oplus X^*(S^{\text{dér}})\otimes \RR $ (voir aussi \cite{BTII}, 4.2.16.). Le noyau de l'application $N(K) \rightarrow \text{Aut}(\ah^1(G,S,K))$ étant $Z(K)_b$, on appelle l'image le groupe de Weyl affine élargi, qui sera noté $W_{\text{af}}^1=N(K)/Z(K)_b$. Lorsqu'on aura discuté le modèle de Néron $\Zh^0$ de $Z$, on parlera aussi du groupe d'Iwahori-Weyl $\widetilde{W}=N(K)/\Zh^0(\oh)$. Il est intéressant de remarquer que, au moins lorsque $G=G^{\text{psréd}}$ est pseudo-réductif, le groupe de Weyl affine $W_{\text{af}}$ se relève en un sous-groupe canonique de $\widetilde{W}$. En effet, on a une identification $\ah(G,S, K)=\ah(\widetilde{G},\widetilde{S},K)$ équivariante par les opérations de $\widetilde{N}(K)$ et $N(K)$, donc il suffit de vérifier que $\widetilde{W}$ s'identifie à $W_{\text{af}}$, lorsque $G=\widetilde{G}$ est modérément universel (parce que alors $\widetilde{W}$ admettrait un morphisme de son cousin modérément universel qui s'identifie à $W_{\text{af}}$ à son tour). En ce cas, on voit facilement que la réunion de $\Zh^0(\oh)$ et des $M_{a,k}$ du 6.2.2. de \cite{BTI} engendre $N(K)$ et que $\Zh^b=\Zh^0$, puisque $Z$ est un produit de groupes de la forme $\Res_{L/K} \GG_m$.

\subsection{} \label{immeubles}
Le but de ce paragraphe est de prolonger l'appartement $\ah(G, S, K)$ associé à un groupe quasi-réductif et quasi-déployé $G$ et à l'un de ses sous-tores déployés maximaux $S$, muni de son opération par $N(K)$, en l'immeuble de Bruhat-Tits $\ih(G, K)$ muni d'une opération par $G(K)$, dont la procédure est grosso modo formelle, lorsqu'on s'est décidé sur ce que doivent être les fixateurs des points $x \in \ah(G, S, K)$ (ceux-ci ne seront pas les sous-groupes parahoriques, en vertu de leur \og non connexité \fg{}). Étant donné une partie non vide et bornée $\Omega$ de $\ah(G, S, K)$, on lui associe une fonction concave à valeurs réelles $f_{\Omega}$ partant de $\Phi$, déterminée par la formule suivante: $f_{\Omega}(a):=\text{inf}\{ k \in \RR: a(\psi-\varphi)+ k \geq 0, \forall \psi \in \Omega \} $. La formule dépende nettement de $\varphi$, mais on omettra ceci lorsqu'aucune confusion n'est à craindre. Enfin le stabilisateur $P_{\Omega}$ associé à $\Omega$ est le sous-groupe de $G(K)$ engendré par le stabilisateur $N_{\Omega}$ de $\Omega$ dans $N(K)$ et le sous-groupe $U_{\Omega}$ engendré par $U_{a, f_{\Omega}(a)}$ pour toutes les racines $a \in \Phi$. Venons finalement à la définition selon le 7.4. de \cite{BTI} de l'immeuble $\ih(G, K)$ associé à la donnée radicielle valuée de $G(K)$:

\begin{defn}
L'ensemble sous-jacent à l'immeuble $\ih(G, K)$ de $G$ est le quotient du produit $G(K) \times \ah(G, S, K)$ par la relation d'équivalence qui identifie $(g_1, \psi_1)$ et $(g_2, \psi_2)$ s'il existe $n \in N(K)$ tel que $n \psi_1=\psi_2$ et $ g_1^{-1}g_2 n \in P_{\psi_1}$. Cet ensemble quotient est muni d'une opération naturelle de $G(K)$ déduite par passage au quotient de l'opération à gauche triviale de $G(K)$ sur le produit ci-dessus et les orbites de l'appartement $\ah(G, S, K) \subseteq \ih(G, K)$, $\psi \mapsto (1, \psi)$ pour cette opération sont les appartements de $\ih(G,K)$ et lui fournissent une structure de complexe simplicial.
\end{defn}

Commençons par discourir sur la apparente dépendance de $\ih(G, K)$ du choix d'un tore $K$-déployé maximal $S$ de $G$ et notons-le provisoirement $\ih(G, S, K)$ pour éviter des confusions ci-dessous. Notons que, si l'on se donne un isomorphisme $\alpha: (G, S) \rightarrow (G', S')$ de couples constitués d'un $K$-groupe quasi-réductif et quasi-déployé et d'un $K$-sous-tore déployé maximal, alors on peut transporter la structures d'une donnée radicielle valuée d'un membre à l'autre par $\alpha$, donc en particulier on arrive à un isomorphisme $\alpha$-équivariant $\ih(G, S, K) \cong \ih(G', S', K)$. Compte tenu que les $K$-sous-tores déployés maximaux de $G$ sont conjugués, on achève un isomorphisme $\ih(G, S, K) \cong \ih(G, gSg^{-1}, K)$ qui est équivariant par les opérations canoniques de $G(K)$ liées par la conjugaison $\text{int}(g): G(K) \rightarrow G(K)$. En composant cet isomorphisme à la droite avec la bijection $g^{-1}:\ih(G, S, K) \rightarrow \ih(G, S, K), h\psi \mapsto g^{-1}h\psi$ qui est équivariante par $\text{int}(g^{-1})$, on achève un isomorphisme $G(K)$-équivariant d'immeubles $\ih(G, S, K)$ et $\ih(G, gSg^{-1}, K)$, qui permet d'identifier $\ah(G, gSg^{-1}, K)$ à $g\ah(G, S, K)$ dans $\ih(G, S, K)$. Par suite, le $G(K)$-ensemble $\ih(G, K)$ ne dépend pas vraiment de $S$ (raison pour laquelle on revient à la notation initiale) et ses appartements sont en correspondance bijective avec les sous-tores $K$-déployés maximaux de $G$. On voit aussi que $P_{\Omega}$ est le stabilisateur de $\Omega$ dans $G(K)$ pour chaque partie $\Omega$ contenue dans un appartement, donc ne dépendant que de $\Omega$ en tant que partie de $\ih(G, K)$ (voir aussi la prop. 7.4.8. de \cite{BTI}). Il y a aussi des autres sous-groupes associés à $\Omega$, dont l'on aura besoin : le fixateur connexe $P_{\Omega}^0$ (resp. le fixateur intermédiaire $P_{\Omega}^1$, resp. le fixateur borné $\hat{P}_{\Omega}$, resp. le stabilisateur borné $P_{\Omega}^{\dagger}$) obtenu en substituant $N^0_{\Omega}:=\Zh^0(\oh)(U_{\Omega}\cap N(K))$ (resp. $N^1_{\Omega}:=Z(K)_b(U_{\Omega}\cap N(K))$, resp. le fixateur $\hat{N}_{\Omega}$ de $\Omega$ dans $N(K)^1$, resp. le stabilisateur $N_{\Omega}^{\dagger}$ de $\Omega$ dans $N(K)^1$) à $N_{\Omega}$ dans la définition de $P_{\Omega}$. De plus, on peut aussi élargir l'immeuble, en posant $\ih^1(G, K)=\ih(G, K) \times X_*(S^c)\otimes {\RR}$, sur lequel $g \in G(K)$ opère par $g\cdot(x, v)=(gx, v+\theta(g))$, et dont les appartements s'identifient à $\ah^1(G,S,K)$, où $S$ décrit l'ensemble des $K$-sous-tores déployés maximaux de $G$. L'un des avantages de ce point de vue est que cet immeuble élargi induit la même bornologie sur $G(K)$ que la valuation $\omega$, d'après le cor. 8.1.5. de \cite{BTI} et la prop. \ref{modele de neron} (comparer avec la prop. 4.2.19. de \cite{BTII}). Pour finir, on remarque que la construction est naturelle pour les extensions séparables $K'/K$ de corps discrètement valués, henséliens, entièrement excellents et résiduellement parfaits, c'est-à-dire, on a une application $\ih(G,K) \rightarrow \ih(G,K')$, qui se factorise à travers les $\Gamma$-invariants si $K'/K$ est galoisienne de groupe $\Gamma$.

Pour conclure, nous observons que nos immeubles de Bruhat-Tits coïncident avec ceux $\ih_{\text{Sol}}(G,K)$ construits à la manière de M. Solleveld \cite{Sol}, qui sont définis comme un produit d'immeubles associés à des groupes réductifs convenables, conformément à la classification par J. Tits des immeubles affines épais de rang $\geq 4$, cf. \cite{TitsImAff}.

\begin{propn}\label{compatibilite avec solleveld}
Il existe un et un seul isomorphisme $G(K)$-équivariant $\ih(G,K) \cong \ih_{\emph{Sol}}(G,K)$ d'immeubles.
\end{propn}

\begin{proof}
Rappelons la définition de l'immeuble $\ih_{\text{Sol}}(G,K)$. On prendre la donnée de présentation $(K'/K, G', Z', Z)$ de $G$ et l'on pose $ \ih(G, K):=\prod_{i\in I}\ih(\bar{G}'_i,\bar{K}'_i)$,
où $\bar{G}'_i$ est le groupe déployé sur $\bar{K}'_i$ ayant les mêmes points rationnels que $G'_i$, cf. les prop. 7.3.3 (2) et 9.9.4 (1) de \cite{CGP}. Le groupe $G(K)$ des points rationnels y opère en conjuguant l'action naturelle de $\widetilde{G}(K)=\prod \bar{G}'_i(\bar{K}'_i)$ et de $Z(K)\rightarrow Z^{\text{aut}}(K) \subseteq \bar{Z}^{\text{ad}}(\bar{G}'_i)$.

 Retournant à notre cadre, on avait déjà vu qu'il y a des isomorphismes canoniques $\ah(\widetilde{G}, \widetilde{S}, K) \cong \ah(G, S, K)$, respectant les opérations des groupes correspondants et l'on vérifie sans peine qu'ils se rassemblent en un isomorphisme équivariant $\ih(\widetilde{G}, K) \cong \ih(G, K)$. Il ne reste qu'à constater que, si $G$ est exotique basique et pseudo-déployé ou absolument non réduit, alors $\ih(G, K) \cong \ih(\bar{G}, K)$, ce qu'on pourrait déduire aisément de l'égalité de certaines $BN$-paires convenables dans $G(K)=\bar{G}(K)$, mais on va montrer cette affirmation pour convaincre le lecteur. Construisons d'abord la bijection entre $\ah(G, S, K)$ et $\ah(\bar{G}, \bar{S}, K) $ lorsque $G$ est exotique basique et pseudo-déployé : si $\psi$ est une valuation de la donnée radicielle abstraite $(Z(K), U_a(K))_{a \in \Phi}$ équipollente à $\varphi$, alors on considère la valuation $\bar{\psi}$ de $(\bar{S}(K), \bar{U}_{\bar{a}}(K))_{\bar{a} \in \bar{\Phi}}$ donnée par $\bar{\psi}_{\bar{a}}(\bar{x}_a(r))=[K_a:K]\psi_a(x_a(r^{\frac{1}{[K_a:K]}}))$. Il est immédiat que $\bar{\varphi}$ est la valuation canonique associée au système de Chevalley $\bar{x}_{\bar{a}}$ pour $\bar{a} \in \bar{\Phi}$ et aussi que, si l'on suppose que $\psi=\varphi+v$ avec $v \in X_*(S)\otimes \RR$, alors $\bar{\psi}_{\bar{a}}-\bar{\varphi}$ est constante égale à $[K_a:K]a(v)$, ce qui coïncide avec $\bar{a}(v)$ moyennant l'identification $X^*(S)\otimes \RR \cong X^*(\bar{S}) \otimes \RR$ (voir la prop. 7.1.5 de \cite{CGP}), d'où la bijection désirée entre les appartements qui commute avec les opérations de $N(K)=\bar{N}(K)$. De plus, on a aussi $P_{\psi}=P_{\bar{\psi}}$ en tant que sous-groupes de $G(K)=\bar{G}(K)$ et par conséquent un isomorphisme équivariant $\ih(G, K) \cong \ih(\bar{G}, K)$. Il en est de même des groupes absolument non réduits : l'isomorphisme $\ah(\text{BC}_n, S, K) \cong \ah(\text{Sp}_{2n}, D_n, K^{1/2})$ se définit ensemblistement en posant $\overline{\varphi}_{a}:=\varphi_a$ si $a$ n'est pas divisible dans $\Phi$ ou $2\varphi_{\frac{a}{2}}$ sinon ; on voit de même que les racines affines, les lois d'opération de $N(K)=\overline{N}(K^{1/2})$ et les sous-groupes parahoriques de $\text{BC}_n(K)=\text{Sp}_{2n}(K^{1/2})$ se correspondent, d'où l'identification $\ih(\text{BC}_n, K)\cong\ih(\text{Sp}_{2n}, K^{1/2})$.
 
 L'unicité de l'isomorphisme équivariant obtenu résulte du 4.2.12. de \cite{BTII}.
\end{proof}

\subsection{} \label{decompositions de bruhat et application de kottwitz}
 
 Dans ce paragraphe, on va signaler quelques faits combinatoires très requis lorsqu'on s'occupe de la théorie de Bruhat-Tits, comme les décompositions de Bruhat. Comme d'habitude, $K$ désignera un corps discrètement valué, hensélien, à anneau d'entiers excellent et à corps résiduel parfait.
 
 \begin{propn}
 Soient $G$ un $K$-groupe quasi-réductif et quasi-déployé et $\emph{\textbf{f}}$, $\emph{\textbf{f}}'$ deux facettes contenues dans un même appartement $\ah(G,S,K)$ de l'immeuble $\ih(G,K)$. Alors, on a un isomorphisme d'ensembles de doubles classes $P^?_{\emph{\textbf{f}}}\setminus G(K)/P^?_{\emph{\textbf{f}}'} \cong N^?_{\emph{\textbf{f}}}\setminus N(K)/N^?_{\emph{\textbf{f}}'}$, où $?=0,1,\hat{}, \dagger$.
 \end{propn}
 
 \begin{proof}
 On peut procéder comme dans \cite{BTI}, 7.3.4. ou \cite{HR}, prop. 8.
 \end{proof}
 
 Rappelons maintenant qu'on a une extension canonique de $W_{\text{af}}$ par $Z(K)_b$ en tant que sous-groupe $N_{\text{af}}^1$ de $N(K)$. D'ailleurs, si $G$ est pseudo-réductif, on trouve même un sous-groupe $N_{\text{af}}^0$ de $N(K)$ extension de $W_{\text{af}}$ par $\Zh^0(\oh)$.
 
 \begin{propn} 
 Gardons les notations antérieures et fixons une alcôve $\emph{\textbf{a}}$ de $\ah(G,S,K)$. Alors, $P^1_{\emph{\textbf{a}}}$ (resp. $P^0_{\emph{\textbf{a}}}$ si $G$ est pseudo-réductif) et $N_{\emph{af}}^1$ (resp. $N_{\emph{af}}^0$ si $G$ est pseudo-réductif) définissent une BN-paire de type $W_{\emph{af}}$ dans le sous-groupe de $G(K)$ qu'ils engendrent.
 \end{propn}
 
 \begin{proof}
 On va établir l'axiome (T 3) des systèmes de Tits de \cite{BTI}, 1.2.6. et on laisse au lecteur le soin d'en compléter. Comme on verra plus tard sans aucune circularité échéante, cf. th. \ref{th. paraboliques et parahoriques}, le sous-groupe parahorique $P^0_{\textbf{f}_i}$ s'écrit comme $P^0_{\textbf{a}}N^0_{\textbf{f}_i}U_{\alpha_i}$, où l'on note $\textbf{f}_i$ une cloison de l'alcôve $\textbf{a}$, $s_i$ la réflexion correspondante et $\alpha_i$ la seule racine affine contenant $\textbf{a}$ telle que $\partial \alpha_i=\textbf{f}_i$. Évidemment il en est de même de $P_{\textbf{f}_i}^1$ et l'on a $P^1_{\textbf{f}_i}w  \subseteq P^1_{\textbf{a}}wP^1_{\textbf{a}} \cup P^1_{\textbf{a}}s_i wP^1_{\textbf{a}}$, ce qu'on vérifie facilement. En effet, on peut substituer $s_iw$ à $w$ si nécessaire sans rien changer, pour que $w^{-1}\alpha_i \supseteq \textbf{a}$, ce qui rende les calculs triviaux. 
 \end{proof}
 
 On note $G(K)^0$ le groupe engendré par les fixateurs connexes $P_{\textbf{f}}^0$ de $G(K)$. Dans le cas pseudo-réductif on obtient la proposition suivante, cf. \cite{HR} :
 
 \begin{propn}\label{parahoriques a la kottwitz}
 Étant donné un groupe pseudo-réductif quasi-déployé $G$ sur $K$ et une facette $\emph{\textbf{f}}$ de $\ih(G,K)$, on a $P_{\emph{\textbf{f}}}\cap G(K)^0=P^0_{\emph{\textbf{f}}}$.
 \end{propn} 
 
 \begin{proof}
 Notons d'abord que si $G=\widetilde{G}$ est modérément universel, alors la première affirmation est bien connue et résulte du th. 6.5.1 de \cite{BTI}, puisque $G(K)^0=G(K)$ et $P_{\textbf{f}}=P^0_{\textbf{f}}$ dans le cas \og simplement connexe \fg{}. En général, on considère l'écriture de $G=(\widetilde{G} \rtimes Z)/\widetilde{Z}$ et l'on observe que $G(K)^0$ est engendré par $\Zh(\oh)$ et l'image de $\widetilde{G}(K)$. Or, si $g \in P_{\textbf{f}}\cap G(K)^0$, on peut écrire $g=[\widetilde{g}]z$, avec $\widetilde{g} \in \widetilde{G}(K)$ et $z \in \Zh(\oh)$. Par hypothèse, $[\widetilde{g}]=gz^{-1}$ stabilise $\textbf{f}$, d'où l'appartenance de $\widetilde{g}$ à $P_{\textbf{f}}^0$ d'après le cas précédent. 
 \end{proof}
 Prenons enfin l'opportunité pour énoncer une conséquence schématique de ce qui précède, en termes des schémas en groupes immobiliers construits dans la suite (sans en recourir).
 \begin{cor}
 Supposons $\kappa$ algébriquement clos et soit $\gh$ est un modèle lisse de $G$ tel que sa composante neutre $\gh^0=\gh^0_{\Omega}$ soit immobilière, alors $\gh$ s'identifie à un sous-$\oh$-groupe ouvert de $\gh_{\Omega}^{\dagger}$.
 \end{cor}
 \begin{proof}
 On peut supposer $\Omega$ clos et fermé et il suffit de montrer que $\gh(\oh) $ le stabilise. Notons que, d'après la conjugaison des tores déployés maximaux et le lemme de Hensel, l'on a $\gh(\oh)=\Nh(\oh)P^0_{\Omega}$, où $\Nh$ désigne le normalisateur de $S$ dans $\gh$. Vu que la partie des points fixés dans $\ah(G,S, K)$ par $P^0_{\Omega}$ coïncide avec $\Omega$ par hypothèse, l'affirmation en découle.
 \end{proof}
 
 \section{Modèles en groupes immobiliers}\label{section geometrique}
 
 Cette section est consacrée à la construction des modèles en groupes attachés aux parties non vides et bornées $\Omega$ de n'importe quel appartement de l'immeuble $\ih(G,K)$ du groupe quasi-réductif $G$ sur le corps fixé $K$. À la fin, on en déduit la descente étale de la théorie de Bruhat-Tits, qui permet de passer du cas quasi-déployé au cas général.
\subsection{} \label{donnees radicielles schematiques}
Dans ce paragraphe on va éclaircir le concept des données radicielles schématiques dû à F. Bruhat et J. Tits, cf. la déf. 3.1.1. de \cite{BTII}, dont on fera usage pour construire les modèles en groupes immobiliers de $G$. Ce paragraphe ressemble le 2.4 de \cite{LouGr}, à la différence près que, dans le présent article, la base est un anneau de valuation discrète.

\begin{defn}[F. Bruhat et J. Tits]
Soient $\oh$ un anneau de valuation discrète, $K$ son corps de fractions, $G$ un $K$-groupe quasi-réductif et $S$ un sous-tore déployé maximal. Une donnée radicielle dans $G$ par rapport à $S$ au-dessus de $\oh$ est la donnée de $\oh$-modèles en groupes lisses, affines et connexes $\Zh$ de $Z$ (resp. $\uh_a$ de $U_a$, où $a \in \Phi_{\text{nd}}$ est non divisible) tels que :
\begin{itemize}[leftmargin=1.5cm]
\item[(DRS 0)] l'injection de $S$ dans $Z$ se prolonge en un morphisme $\sh \rightarrow \Zh$, où $\sh$ désigne le seul groupe diagonalisable à isomorphisme unique près modelant $S$.
\item[(DRS 1)] l'application conjugaison $Z \times U_a \rightarrow U_a$ se prolonge en une application $\Zh \times \uh_a \rightarrow \uh_a$.
\item[(DRS 2)] si $b \neq -a$, l'application commutateur $U_a \times U_b \rightarrow \prod_{c=pa+qb \in \Phi_{\text{nd}}, (p,q) \in \QQ_{>0}^2} U_c$ se prolonge en un morphisme $\uh_a \times \uh_b \rightarrow \prod_{c=pa+qb \in \Phi_{\text{nd}}, (p,q) \in \QQ_{>0}^2} \uh_c$.
\item[(DRS 3)] si on note $W_a$ l'ouvert $U_a \times U_{-a} \cap U_{-a} \times Z \times U_a$ dans $U_a \times U_{-a}$, alors il existe un voisinage ouvert $\wh_a$ de $\uh_a \times \uh_{-a}$ contenant la section unité et modelant $W_a$ tel que l'inclusion $W_a \rightarrow U_{-a}\times Z \times U_a$ se prolonge en une application $\wh_a \rightarrow \uh_{-a} \times \Zh \times \uh_a$.
\end{itemize}
\end{defn} 

Le lecteur aura déjà compris que le but de cette définition est de construire, par \og recollement fermé \fg{} de $\Zh$ et des $\uh_a$, $a \in \Phi_{\text{nd}}$, un certain schéma en groupes $\gh$ au-dessus de $\oh$ avec de bonnes propriétés. Le §3 de \cite{BTII} traite de prouver un tel résultat à l'aide de représentations linéaires entières et fidèles, en exploitant la théorie des représentations de plus grand poids des groupes réductifs connexes (voir \cite{TitsRep}), ce qui n'est pas bien entendu pour les groupes quasi-réductifs. Cependant comme les auteurs leur-mêmes avaient remarqué au 3.1.7. de \cite{BTII}, la méthode de preuve indiquée pour le problème en question consisterait à appliquer la théorie des lois birationnelles du \cite{SGA3}, exp. XVIII, lorsqu'on aura un énoncé du th. 3.7 de \cite{BTII} dans le monde des schémas et pas seulement des espaces algébriques. Entre-temps ceci a été rendu possible par S. Bosch, W. Lütkebohmert et M. Raynaud, voir le th. 6.6.1. de \cite{BLR}, ce qu'on fera dans la suite, en simplifiant quelques idées de \cite{LndvCpc}, §5 et 6.

\begin{thm}[F. Bruhat et J. Tits, à peu près]\label{theoreme d'existence et unicite des groupes attaches aux drs}
Reprenons les notations de la définition précédente. Alors il existe un et un seul $\oh$-groupe affine, lisse et connexe $\gh$ modelant $G$ tel que :
\begin{enumerate}
\item les inclusions $Z \rightarrow G$ et $U_a \rightarrow G$ pour tout $a \in \Phi_{\emph{nd}}$ se prolongent en des isomorphismes de $\Zh$ et $\uh_a$ sur des sous-$\oh$-groupes fermés de $\gh$.
\item pour tout système de racines positives $\Phi^+$ et tout ordre mis sur $\Phi^+$, le morphisme $\prod_{a \in \Phi_{\emph{nd}}^+} \uh_a \rightarrow \gh$ soit un isomorphisme sur un sous-$\oh$-groupe fermé $\uh^+$ de $\gh$.
\item l'application produit $\uh^-\times \Zh \times \uh^+ \rightarrow \gh$ définisse un ouvert du membre de droite.
\end{enumerate}
\end{thm}

\begin{proof}
Lorsqu'on se donne un ordre grignotant sur $\Phi_{\text{nd}}^+$, le 3.3. de \cite{BTII} fournit une loi de groupe dans le produit $\uh^+=\prod_{a \in \Phi_{\text{nd}}^+}\uh_{a}$, en argumentant par récurrence sur la longueur des parties positivement closes $\Psi$ de $\Phi_{\text{nd}}^+$ et en appliquant la règle (DRS 2). Notons également que $\Zh$ opère sur $\uh$ canoniquement d'après la règle (DRS 1), dont la restriction au sous-groupe fermé $\sh$ (voir le cor. 2.5. et le th. 6.8. de \cite{SGA3}, exp. IX) opère sur $\uh_a$ par le poids $a$. Comme les groupes trouvés sont toujours lisses et connexes, le fait que $\uh^+$ s'écrit comme produit de ses sous-groupes fermés $\uh_{a}$ rangés dans un ordre quelconque résulte du théorème principal de Zariski et la même indépendance de l'ordre au-dessus des corps. 

Pour construire $\gh$ satisfaisant à la troisième propriété, il suffit de construire une loi birationnelle sur $\ch_{\Phi^+}:=\uh^-\times \Zh \times \uh^+$ pour un choix quelconque $\Phi^+$ de racines positives, grâce au th. 6.6.1 de \cite{BLR}. En effet, on n'a pas besoin de remplacer la grosse cellule par un ouvert schématiquement dense relativement à $\oh$ d'après 5.1 de \cite{BLR}. Il suffit de définir une application birationnelle $\ch_{\Phi_+} \dashrightarrow \ch_{\Phi_-}$ prolongeant l'identité de $G$ dans la fibre générique, car cela permettra d'obtenir des applications rationnelles de produit et d'inversion dans $\ch_{\Phi^+}$, en commutant tous les facteurs et en multipliant ou inversant lorsque convenable. Montrons plus généralement la même assertion pour chaque paire de systèmes de racines positives $\Phi^+$ et $\Phi_1^+$ par récurrence sur le cardinal de $ \Phi_{\text{nd}}^+ \cap \Phi_{1, \text{nd}}^- $. Si cette intersection ne contient que la racine non divisible $a$, alors l'axiome (DRS 3) en fournit un morphisme birationnel $$\ch_{\Phi^+}=\uh_{\Phi_{\text{nd}}^-\setminus \{-a\}}\times \uh_{-a} \times \Zh \times \uh_a \times \uh_{\Phi_{\text{nd}}^+\setminus \{a\}} \dashrightarrow \uh_{\Phi_{\text{nd}}^-\setminus \{-a\}}\times \uh_{a} \times \Zh \times \uh_{-a}\times  \uh_{\Phi_{\text{nd}}^+\setminus \{a\}}=\ch_{\Phi_1^+}.$$ Ceci montre aussi que le groupe obtenu est indépendant du système de racines positives $\Phi^+$ choisi.

Maintenant on dispose du $\oh$-groupe lisse, séparé et connexe $\gh$ (donc affine grâce à la prop. 12.9 de \cite{SGA3}) qui modèle $G$ et qui contient les différents $\ch_{\Phi^+}$ en tant qu'ouverts, et il faut encore voir que les sous-groupes localement fermés $\Zh$ et $\uh^+$ sont effectivement fermés. Or l'image réciproque de $\Zh\uh^+:=\Zh \ltimes \uh^+$ par le revêtement fidèlement plat $\ch^- \times \ch^+ \rightarrow \gh$ est déterminée par la condition fermée $b_1b_2=1$, où $(a_1,b_1, b_2, a_2) \in \Zh\uh^+ \times (\uh^-)^2 \times \Zh \uh^+$, ce qui achève notre affirmation par descente fidèlement plate.
\end{proof}
 
\subsection{}\label{schemas en groupes immobiliers}

Enfin nous allons construire les $\oh$-modèles en groupes $\gh_{\Omega}$ d'un $K$-groupe quasi-réductif quasi-déployé $G$ associés aux parties bornées d'un certain appartement de $\ih(G,K)$. On construit d'abord le $\oh$-modèle $\Zh$ de $Z$ à l'aide de la théorie des modèles de Néron (localement de type fini) de \cite{BLR}, chap. 10.

\begin{propn}[S. Bosch, W. Lütkebohmert et M. Raynaud, à peu près] \label{modele de neron}
Soient $K$ un corps discrètement valué, $\oh$ son anneau d'entiers supposé hensélien et excellent, et $\kappa$ son corps résiduel supposé parfait. Étant donné un $K$-groupe quasi-réductif et nilpotent $Z$, alors il possède un modèle de Néron $\Zh$ sur $\oh$, dont le sous-groupe ouvert maximal de type fini $\Zh^1$ représente $Z(K)_b=\Zh^1(\oh)$.
\end{propn}

\begin{proof}
La première assertion sur l'existence d'un modèle de Néron est due à S. Bosch, W. Lütkebohmert et M. Raynaud, voir le th. 10.2.1 de \cite{BLR}, dont la démonstration consiste des étapes suivantes : on peut éliminer le tore maximal $T$ de $Z$ en prenant une extension finie qui le déploie et en appliquant les 10.1.4 et 10.1.7 de \cite{BLR} ; lorsque $Z$ est unipotent ployé, on se ramène à montrer que $Z(K^{\text{nr}})$ est borné, ce qu'il faut montrer par dévissage lorsque $Z=Z[p]$ est $p$-torsion ; on utilise l'écriture explicite de ces groupes due à J. Tits pour en construire une compactification sans aucun point rationnel à l'infini, cf. prop. 11 de \cite{BLR} ou th. VI.1. de \cite{Oes}.

Afin de prouver la dernière affirmation, on peut supposer $K=K^{\text{nr}}$ strictement hensélien. Notons que le groupe $\pi_0(\Zh_{s})(\kappa)$ est une extension du groupe fini $\pi_0(\Zh_s/\sh_{s})$ par le groupe abélien de rang $\dimn S$ $\pi_0(\sh_s)$. En même temps, l'application $p: Z(K)\rightarrow X_*(S)\otimes \QQ$ donnée par $\chi(p(z))=\omega(\chi(z))$ se factorise en une application $\bar{p}:\pi_0(\Zh_s)\rightarrow X_*(S)\otimes \QQ$, telle que l'image de $\pi_0(\sh_s)$ en détermine un réseau. Par conséquent, on voit que le noyau de $\bar{p}$ est un groupe fini distingué et égal à $Z(K)_b/\Zh^0(\oh)$ et que tout sous-groupe fini de $\pi_0(\Zh_s)$ en est contenu, ce qui implique l'existence d'un sous-groupe ouvert $\Zh^1$ de $\Zh$ possédant les propriétés désirées.
\end{proof}

Maintenant on se penche sur les modèles entiers $\uh_{a, \Omega}$ des $U_a$, $a \in \Phi$.

\begin{lem}
Soient $k, l \in \RR$ réels quelconques. Alors, il existe un et un seul $\oh$-modèle affine, lisse et connexe $\uh_{a, (k,l)}$ de $U_a$ tel que $\uh_{a, (k,l)}(\oh')=U'_{a,k}U'_{2a, l}$ pour toute extension conservatrice $K'$ de $K$.
\end{lem}

\begin{proof}
L'unicité résulte toujours du fait que ces modèles deviennent immédiatement étoffés, cf. déf. 1.7.1. de \cite{BTII}, dès que le corps résiduel $\kappa'$ est infini. Les modèles lisses, affines et connexes $\uh_{2a,l}$ des sous-groupes radiciels $U_{2a}$ attachés aux racines divisibles seront obtenus comme adhérences schématiques de $U_{2a}$ dans $\uh_{a,(k,l)}$, donc on suppose que $a\in \Phi_{\text{nd}}$ n'est pas divisible. On peut supposer aussitôt que $G=\widetilde{G}^{\text{psréd},a}$ est pseudo-réductif, modérément universel et de rang pseudo-déployé $1$. Lorsque $\text{dom}(\zeta_a)= \Res_{K_a/K} \text{SL}_2$, il résulte que $U_{a,k}$ s'identifie à $\varpi_a^n\oh_{a}$ dedans $K_a\cong U_a(K)$, où $\oh_a$ désigne l'anneau d'entiers de $K_a$, $\varpi_a$ une uniformisante et $n$ est le plus petit entier majorant $[K_a:K]k$. On voit aisément que $U_{a,k}$ est représenté par la restriction des scalaires $\uh_{a,k}$ du spectre affine de $\oh_{a}[\varpi_a^{-n} t]$ le long de l'extension finie et plate $\oh \rightarrow \oh_a$ (en vertu du fait que $\oh$ soit japonais), identifié par transport de structure le long de $x_a$ à un modèle en groupes affine, lisse et connexe de $U_a$, et que ceci commute aux changements de base conservateurs. Ici on utilise l'hypothèse de excellence faite sur $\oh$, pour savoir que $\oh_a$ est toujours fini sur $\oh$.

Si $a$ est multipliable, on peut supposer sans perdre de généralité que $k \in \Gamma_a'$ et $l \leq \text{inf}(\Gamma_{2a} \cap [2k, +\infty[) $. Supposons que $\text{dom}(\zeta_a)=\Res_{K_{2a}/K}\text{SU}_{3,K_{a}/K_{2a}}$ et fixons $\lambda$ de valeur $\omega(\lambda)$ maximal parmi les éléments de trace $1$. Alors, $x_a$ identifie $U_a$ au groupe $\Res_{K_{2a}/K}H^{\lambda}$ de \ref{quasi-epinglages} dont les points rationnels sont les couples $(x,y) \in K_a \times K_a^0$, de telle sorte que la valuation soit donnée par $\varphi_a(x,y)=\frac{1}{2}\omega(y+\lambda x\sigma(x))$. Posant $\gamma=-\frac{1}{2}\omega(\lambda)$, on tire $\Gamma_a'=-\gamma+\omega(K_a^{\times})$, $\Gamma_{2a}=\omega(K_{a}^0\setminus 0)$ et $U_{a,k}U_{2a,l} $ coïncide avec la partie des couples $(x,y)$ tels que $\omega(x)\geq k+\gamma$ et $\omega(y)\geq l$, cf. 4.3.5. et 4.3.7. de \cite{BTII}. Grâce au paragraphe précédent, on obtient un modèle entier en groupes lisse, affine et connexe $\uh_{a, (k,l)}$ de $U_a$ représentant $U_{a,k}U_{2a,l}$, qui commute aux changements de base conservateurs. 

Considérons finalement le cas où $\text{dom}(\zeta_a)=\Res_{K_{2a}/K} \text{BC}_1$. On a une identification $U_a= \Res_{K_{a}/K} \GG_a \times \Res_{K_{2a}/K}\GG_a$, avec $K_a=K_{2a}^{1/2}$, de telle manière que $\varphi_a(x,y)=\frac{1}{2}\omega(\alpha x^2+y)$ ; en particulier, on a $\Gamma_{2a}=\omega(K_{2a}^{\times})$ et $\Gamma_a'=\frac{1}{2}\omega(K_a^{\times})\setminus \omega(K_a^{\times})$. Choisissons $\alpha$ tel que $\omega(\alpha) \notin \omega(K_{2a}^{\times})$ pour que $\omega(\alpha x^2+y)=\text{inf}\{\omega(\alpha)+2\omega(x), \omega(y)\}$ et il suffit maintenant d'observer que $U_{a,k}U_{2a,l}$ est formé par la partie des couples $(x,y)$ tels que $\omega(x)\geq k+\gamma$ et $\omega(y)\geq l$, où $\gamma=-\frac{1}{2}\omega(\alpha)$, laquelle est représentable par un $\oh$-modèle en groupes lisse, affine et connexe $\uh_{a, (k,l)}$ de $U_a$.
\end{proof}

Soit $\uh_{a,\Omega}:=\uh_{a, (f_{\Omega}(a), f_{\Omega}(2a))}$ le $\oh$-modèle lisse, affine et connexe de $U_a$ dont le groupe des points entiers $\uh_{a,\Omega}(\oh')$ coïncide avec le sous-groupe $U'_{a,f(a)} U'_{2a, f(2a)}$ pour toute extension conservatrice $K'$ de $K$. Toutefois, par une question de convenance, nous travaillerons avec les fonctions $f: \Phi \rightarrow \RR$ quasi-concaves au sens de la déf. 4.5.3. de \cite{BTII} (voir aussi la déf. 6.4.7. de \cite{BTI}, ainsi que l'addendum E2 de \cite{BTII}), plutôt que avec les parties bornées et non vides $\Omega$ des appartements, fonctions lesquelles ont été dessinées pour assurer qu'on puisse constituer des données radicielles schématiques à leur aide. Quelques énoncés se simplifient lorsqu'on considère l'optimisée $f'$ de $f$ donnée par $ f'(a):=\text{inf}\{ k \in \Gamma_a': k\geq f(a) \text{ ou, lorsque }\frac{a}{2}\in \Phi, k \geq 2f(\frac{a}{2}) \}
$ - cela ne change grand-chose, car on a $\uh_{a,f}=\uh_{a,f'}$, compte tenu de la rem. 6.4.12.(a) de \cite{BTI}. Remarquons que l'on possède de même des sous-groupes $P^0_f$, $P^1_f$ (et même $P^{\dagger}_f$, cf. 4.6.16 et 4.6.17. de \cite{BTII}) associés à $f$ en reprenant les définitions du \ref{immeubles} avec $f$ au lieu de $f_{\Omega}$.

\begin{thm}
Gardant les notations ci-dessus, les $\oh$-groupes $\Zh^0$ et $\uh_{a, f}$, pour toute racine non divisible $a \in \Phi_{\emph{nd}}$, définissent une donnée radicielle schématique dans le groupe quasi-réductif et quasi-déployé $G$. Le schéma en groupes $\gh_f^0$ fourni par le th. \ref{theoreme d'existence et unicite des groupes attaches aux drs} est l'unique $\oh$-groupe affine, lisse et connexe, satisfaisant à $\gh^0_f(\oh')=P_{f,K'}^0$ pour toute extension conservatrice $K'$ de $K$.
\end{thm}

\begin{proof}
Afin de vérifier les règles, remarquons qu'on peut, grâce à un changement de base conservateur, supposer le corps résiduel parfait $\kappa$ infini : si $f$ est l'optimisée d'une fonction concave, cela est évident ; sinon, on a besoin de faire des calculs comme dans la prop. 4.5.5., le cor. 4.5.7. et la prop. 4.5.10. de \cite{BTII}. En particulier, les $\oh$-groupes $\sh$, $\Zh^0$ et $\uh_{a,f}$ sont supposés étoffés, cf. déf. 1.7.1. de \cite{BTII}. Ceci étant, l'axiome (DRS 0) résulte de la fonctorialité des modèles de Néron connexes, (DRS 1) de la compatibilité des données radicielles valuées et (DRS 2) de la propriété (QC 2) des fonctions quasi-concaves. Quant à (DRS 3), il faut premièrement définir la fonction $d_a\in K[U_{-a}\times U_a]$ dont l'ouvert spécial détermine l'intersection de $U_{-a}\times U_a$ avec la grosse cellule $C=U^+\times Z \times U^-$ : on se ramène rapidement au cas des groupes de rang pseudo-déployé $1$, à savoir les restrictions des scalaires de $\text{SL}_2$, $ \text{SU}_3$ et $\text{BC}_1$ ; les deux premiers cas se traitent comme dans 4.1.6. et 4.1.12. de \cite{BTII} et le dernier y se ramène en tirant en arrière la fonction $d_a$ pour $\text{SL}_2$. Ensuite on procède verbatim comme dans 4.5.8. de \cite{BTII}, à savoir on vérifie que cette fonction est entière par un calcul de points entiers et l'on note $\wh_{a,f}$ l'ouvert spécial défini par $d_a \in \oh[\uh_{-a,f} \times \uh_{a,f}]$, qui contient la section unité et est aussi étoffé d'après 1.7.3.e de \cite{BTII}. Enfin il ne reste qu'à voire que $\beta_a: W_{a} \rightarrow U_a \times Z \times U_{-a}$ se prolonge en un morphisme entier. Il est aisé de calculer que $\text{pr}_{U_{\pm a}}\circ \beta_a(\wh_{a,f}(\oh))\subseteq \uh_{\pm a, f}(\oh)$, en se ramenant au cas pseudo-réductif et modérément universel de rang pseudo-déployé égal à $1$. D'autre part, le morphisme composé de schémas $\text{pr}_z\circ \beta_a$ se prolonge en un morphisme entier $\wh_{a,f} \rightarrow \Zh$ grâce à la propriété universelle du modèle de Néron de $Z$, qui, en fait, se factorise à travers de $\Zh^0$ par une question de connexité du membre de gauche et ce que l'image de la section unité de $\wh_{a,f}$ est la section unité de $\Zh$. Enfin la représentabilité de $P^0_{f}$ par $\gh_{f}$ découle d'un raisonnement classique avec des ouverts denses, voir le cor. 4.6.7. de \cite{BTII}.
\end{proof}

Retournant aux parties bornées non vides $\Omega$ d'un appartement quelconque $\ah(G, S, K)$ de $\ih(G, K)$, on note $\gh_{\Omega}:=\gh_{f_{\Omega}}=\gh_{f'_{\Omega}}$ le $\oh$-modèle en groupes de $G$ déduit du th. précédente et on l'appelle le modèle immobilier attaché à $\Omega$. Si $\Omega$ est, en plus, une facette, alors $\gh_{\textbf{f}}$ est dénommé un modèle parahorique. On remarque encore qu'on possède aussi d'autres $\oh$-modèles en groupes $\gh_{\Omega}^1$ (resp. $\hat{\gh}_{\Omega}$, resp. $\gh_{\Omega}^{\dagger}$) de $G$ représentant $P_{\Omega}^1$ (resp. $\hat{P}_{\Omega}$, resp. $P_{\Omega}^{\dagger}$), mais cependant on ne sait plus que pour $P_{\Omega}^1$, si cette représentabilité est préservée lorsqu'on prendre d'extensions conservatrices et même non ramifiées. Si l'on ne se intéressait pas à la grosse cellule, on pourrait avoir construit les modèles entiers $\gh^0_{\Omega}$ de manière plus commode, cf. lem. 7.2 de \cite{Yu}, en prenant l'adhérence schématique de $G$ dans une représentation convenable $\text{GL}_n$ et puis en lissifiant la composante neutre du modèle plat échéant au moyen d'une suite de dilatations (voir soit \cite{WW}, soit le 3.2 de \cite{BLR}, soit \cite{RichDil}).

M. Solleveld nous a demandé si l'existence de bonnes filtrations au sens de P. Schneider et U. Stuhler \cite{SchSt} de $Z(K)$ par rapport à $G(K)$ peut être déduite de nos modèles entiers. Bien que l'on a toujours la filtration exhaustive des sous-groupes de congruence modulo $\varpi^n$, celle-ci est typiquement trop faible ; d'autre part, on pourrait définir la filtration de Moy-Prasad dans ce cadre, mais elle ne serait pas toujours exhaustive, parce que $Z(K)$ n'est pas forcément sans $p$-torsion, même lorsque $G$ est pseudo-réductif, cf. \cite{CGP}, ex. 1.6.3. Cependant nous ne savons pas comment unifier les deux méthodes. En théorie des représentations $p$-adiques, on s'intéresse aussi aux schémas en groupes associés aux fonctions quasi-concaves $f: \Phi \cup \{0\} \rightarrow \RR$. 

\subsection{}\label{parahoriques et paraboliques}
Reprenons les notations des paragraphes précédents et déterminons maintenant le radical unipotent du $k$-groupe $\gh^0_{f,s}:=\gh^0_f \otimes_{\oh} k$, où $f:\Phi \rightarrow \RR$ est une fonction quasi-concave et optimale. On note $f^*: \Phi \rightarrow \RR$ la fonction qui fait correspondre à une racine $a$ le plus petit nombre réel appartenant à $\Gamma_a'$ et majorant strictement $f(a)$ si $f(a)+f(-a)=0$ ou tout simplement la même valeur $f(a)$ si $f(a)+f(-a)>0$.

\begin{lem}\label{lemme du radical unipotent}
La fonction $f^*$ est quasi-concave, le radical unipotent de $\gh^0_{f,s}$ s'identifie à l'image de $\uh^+_{f^*,s}\times R_u\Zh^0_{s}\times \uh^-_{f^*,s}$ par le morphisme produit et le système de racines du quotient réductif $\gh^{0, \emph{réd}}_{f,s}$ suivant le tore déployé maximal $\sh_s$ s'identifie à $\Phi_f=\{ a \in \Phi: f(a)+f(-a)=0\}$.
\end{lem}

\begin{proof}
On peut supposer sans perdre de généralité que $\kappa$ est infini, en prenant une extension conservatrice non ramifiée et en appliquant la descente étale. Vu que la fibre spéciale $\Zh^0_s$ (resp. $\uh_{a,f,s}$) coïncide avec le centralisateur du tore déployé maximal $\sh_s$ dans $\gh^0_{f,s}$ (resp. le sous-groupe radiciel associé au poids $a \in \Phi_{\text{nd}}$), cf. 4.6.4 de \cite{BTII}, on peut d'après soit \cite{BTII}, 1.1.11, soit la prop. 3.3.10 de \cite{CGP}, se ramener au cas où $\Phi$ consiste d'un seul rayon radiciel. Maintenant on observe que $\beta_a$ se prolonge en une application de $\uh_{-a, f^*}\times \uh_{a, f}$ dans $\uh_{a,f}\times \Zh^0 \times \uh_{-a, f^*}$, dont les calculs ressemblent ceux de 4.5.8 ii de \cite{BTII}. Par suite, la partie localement fermée $\uh_{a,f,s}\times \Zh_s^0\times \text{im}\,\uh_{-a,f^*,s}$ de $\gh^0_{f,s}$ est même un sous-groupe fermé et résoluble, ce qui termine la démonstration lorsque $f=f^*$. Supposons que $f^*(a)>f(a)$ de telle sorte que $f(a)+f(-a)=0$, et prenons un élément $u \in U_a$ tel que $\varphi_a(u)=f(a)$ . Alors $\varphi_{-a}(u')=\varphi_{-a}(u'')=-f(a)$, où $m(u)=u'uu''$ est le seul élément de $N(K)$ s'écrivant de cette manière, donc $m(u) \in \gh^0_{f,s}$ y induit la réflexion de $\sh_s$ dans $\gh^{0, \text{réd}}_{f,s}$ par rapport à $a$, en échangeant le sous-groupe résoluble $\uh_{a,f,s}\times \Zh_s^0\times \text{im}\,\uh_{-a,f^*,s}$ par son opposé $\text{im}\,\uh_{a,f^*,s}\times \Zh_s^0\times \uh_{-a,f,s}$. On note que cet argument ne change pas lorsque $f^*(2a)>f(2a)$ et qu'il implique que le complément de $ \text{im}\, \uh_{a,f^*,s}(\kappa)$ dans $\uh_{a,f,s}(\kappa)$ n'appartient à $R_u\gh^0_{f,s}$. Par densité des points rationnels lorsque $\kappa$ est infini et parfait, il s'ensuit que le sous-groupe radiciel $V_{a,f}$ de $\gh^{0, \text{réd}}_{f,s}$ s'identifie à $\uh_{a,f,s}/\text{im}\,\uh_{a,f^*,s}$, d'où la conclusion. Les affirmations concernant la quasi-concavité de $f^*$ et le système de racines du quotient réductif en résultent.
\end{proof}

 Enfin on arrive à la généralisation quasi-réductive de l'énoncé du théorème fondamental 4.6.33 de \cite{BTII} sur le lien entre les sous-groupes parahoriques $P_{\textbf{f}}$ attachés à une facette $\textbf{f}$ de $\ih(G, K)$ et les sous-groupes paraboliques du quotient réductif de la fibre spéciale $\gh_{\textbf{f},s}$.

\begin{thm}\label{th. paraboliques et parahoriques}
Reprenons toutes les notations introduites aux paragraphes précédents et rappelons que $G$ est supposé quasi-déployé. Soit $\emph{\textbf{f}} \subseteq \ih(G, K)$ une facette et $\ih(G,K)_{\emph{\textbf{f}}}$ l'étoile de $\emph{\textbf{f}}$, c'est-à-dire, l'ensemble ordonné de toutes les facettes $\emph{\textbf{f}}'$ dont l'adhérence contient $\emph{\textbf{f}}$. À chaque membre $\emph{\textbf{f}}'$ de l'étoile, associons le $\kappa$-sous-groupe algébrique  $Q_{\emph{\textbf{f}}', \emph{\textbf{f}}}$ de $\gh^0_{\emph{\textbf{f}},s}$
image du morphisme canonique $ \gh^0_{\emph{\textbf{f}}',s}\rightarrow \gh^0_{\emph{\textbf{f}},s} $. Les propriétés suivantes sont remplies:

\begin{enumerate}
\item Cette application $\emph{\textbf{f}}' \mapsto Q_{\emph{\textbf{f}}', \emph{\textbf{f}}}$ est un isomorphisme d'ensembles ordonnés entre l'étoile de $\emph{\textbf{f}}$ et l'ensemble des sous-groupes paraboliques de $\gh^0_{\emph{\textbf{f}},s}$ ordonné par la relation opposée à l'inclusion. 

\item Le sous-groupe parahorique $\gh^0_{\emph{\textbf{f}}'}(\oh) \subseteq  \gh^0_{\emph{\textbf{f}}}(\oh) $ est égal à l'image réciproque de $Q_{\emph{\textbf{f}}', \emph{\textbf{f}}}(\kappa)$ par la projection $\gh^0_{\emph{\textbf{f}}}(\oh) \rightarrow \gh^0_{\emph{\textbf{f}},s}(\kappa)$.
\end{enumerate}
\end{thm}

\begin{proof}
Soient $\ah(G,S, K)$ un appartement de l'immeuble contenant $\textbf{f}$, fixé dans la suite, et $\textbf{a}$ l'unique alcôve dont l'adhérence contient $\textbf{f}$ telle que $a(x)+f_{\textbf{f}}(a) >0$ pour tout $x \in \textbf{a}$ et toute racine positive $a \in \Phi_{\textbf{f}}^+$ du sous-système de racines de $\textbf{f}$. On va premièrement étudier les facettes de l'appartement emboîtées entre $\textbf{f}$ et $\textbf{a}$, vérifier qu'elles correspondent par l'application de l'énoncé aux sous-groupes paraboliques standards de $\gh^0_{\textbf{f},s}$ contenant le sous-groupe pseudo-parabolique minimal attaché aux racines positives $\Phi_{\textbf{f}}^+$, ainsi que les restantes affirmations du théorème. 

Observons que ces facettes emboîtées correspondent bijectivement aux parties $J$ de $\Delta_{\textbf{f}}$, en envoyant $\textbf{f}'$ sur l'ensemble $J= \{a \in \Delta_{\textbf{f}}: f_{\textbf{f}'}(-a)=f_\textbf{f}(-a)\}$, et notons que $\Phi_{\textbf{f}_J}$ est égal au sous-système de racines de $\Phi_{\textbf{f}}$ engendré par $J$. Comme évidemment $f_{\textbf{f}'} \geq f_{\textbf{f}}$ pour toute sur-facette de $\textbf{f}$, il résulte qu'il y a une application canonique $\gh^0_{\textbf{f}_J} \rightarrow \gh^0_{\textbf{f}}$, dont l'image s'identifiera selon l'énoncé au sous-groupe parabolique standard de type $J$ de $\gh^0_{\textbf{f}, s}$. Pour le voir, observons que, si la somme $f_{\textbf{f}}(a) + f_{\textbf{f}}(-a)>0$ est strictement positive, alors elle est égale à la plus grand longueur des intervalles ouverts de $\RR \setminus \Gamma_a'$. Par suite, $f_{\textbf{f}_J}(a) =f_{\textbf{f}}(a) $ pour toute racine $a \notin \Phi_{\textbf{f}}$ et, d'autre part, par hypothèse $ f_{\textbf{f}_J}(a)$ est égale à $f_{\textbf{f}}(a)$ si $a \in \Phi_{\textbf{f}}^+ \cup \Phi_{\textbf{f}_J}$ et à son successeur dans $\Gamma_a'$ sinon. De ces considérations et du lem. \ref{lemme du radical unipotent}, il découle que $Q_{\textbf{f}_J,\textbf{f}}$ est justement le sous-groupe pseudo-parabolique de type $J$. On voit de même, compte tenu de la surjectivité de $\uh_{a, \textbf{f}}(\oh) \rightarrow \uh_{a, \textbf{f}}(\kappa)$ et $\Zh^0(\oh) \rightarrow \Zh^0(\kappa)$, vu que $\oh$ est hensélien, que l'application $\gh^0_{\textbf{f}_J}(\oh) \rightarrow Q_{\textbf{f}_J, \textbf{f}}(\kappa)$ est surjective. Pour régler la dernière affirmation, il faut tout simplement remarquer que le noyau de la projection $\gh^0_{\textbf{f}}(\oh) \rightarrow \gh^0_{\textbf{f},s}(\kappa)$ est contenu dans $\gh^0_{\textbf{f}'}(\oh)$, d'après le 4.6.8 (i) de \cite{BTII}.

Le cas général se déduit de ce qui précède par conjugaison, en observant que, d'une part, $N_{\textbf{f}}^0$ permute transitivement les alcôves qui contiennent $\textbf{f} $ et le sous-groupe $U_{\textbf{f}}$ engendré par les $U_{a, \textbf{f}}$ pour $a \in \Phi$ permute transitivement les appartements contenant $\textbf{f}$ (voir la prop. 7.4.6. de \cite{BTI}), tandis que, d'autre part, les $\kappa$-sous-groupes paraboliques de $\gh^0_{\textbf{f},s}$ sont conjugués d'un standard par un certain élément de $\gh^0_{\textbf{f},s}(\kappa)$. L'injectivité résulte de ce que chaque paire de facettes sont contenues dans un même appartement $\ah(G,S,K)$, tandis que la partie des points fixes de $\gh^0_{\textbf{f}}(\oh)$ dans celui-ci est égal à l'adhérence de $\textbf{f}$ (voir \cite{BTI}, prop. 7.4.5.).
\end{proof}

\subsection{} \label{descente etale}
Maintenant on va descendre l'immeuble le long d'une extension non ramifiée à la suite de \cite{BTII}, V (plus récemment, G. Prasad a trouvé dans \cite{PrDesc} une stratégie différente pour y parvenir, ce qui ne nous aide pas forcément dans notre cadre non réductif). Soient $K$ un corps discrètement valué et hensélien, $\oh$ son anneau d'entiers supposé excellent, $\kappa$ son corps résiduel supposé parfait, et $G$ un groupe quasi-réductif quelconque sur $K$. Si l'on note $K^{\text{nr}}$ la plus grande extension algébrique non ramifiée de $K$, alors $G$ devient quasi-déployé sur $K^{\text{nr}}$, cf. cor. \ref{pseudo-steinberg}, qui on peut appliquer grâce au fait que $K^{\text{nr}}$ est un corps $C1$ (voir \cite{Lan52}), et on dispose donc d'un immeuble $\ih(G, K^{\text{nr}})$, qui est opéré par $\Gamma:=\text{Gal}(K^{\text{nr}}/K)$ de façon compatible avec l'opération naturelle de ce dernier sur $G(K^{\text{nr}})$. La définition naïve de l'immeuble $\ih(G,K)$ serait donc la partie $\ih(G, K^{\text{nr}})^{\Gamma}$ et on définirait les modèles en groupes parahoriques $\gh^0_{\Omega}$ associés à des parties de $\ih(G, K)$ comme les descendus étales de tels modèles définis au-dessus de $\oh^{\text{nr}}$. Cependant pour qu'on déduise de bonnes propriétés géométriques et combinatoires de ces objets, il faut travailler encore comme dans le § 5. de \cite{BTII}. Pour être plus précis, on aura besoin des propriétés combinatoires de l'immeuble $\ih(G, K^{\text{nr}})$ joint à l'opération de $\Gamma$ pour déduire l'existence d'une valuation de la donnée radicielle abstraite $(Z(K), U_a(K))_{a \in \Phi}$ dans $G(K)$, où $S$ est un tore déployé maximal de $G$, $Z=Z_G(S)$ est le centralisateur de $S$ et $U_a$ est le sous-groupe radiciel attaché à la racine $a \in \Phi:=\Phi(G,S)$ ; l'immeuble résultant sera identifié à la fin avec les points fixes $\ih(G, K^{\text{nr}})^{\Gamma}$.

\begin{lem}\label{tore breve-deploye defini sur corps base}
Il existe un tore $K^{\emph{nr}}$-déployé maximal $S^{\emph{nr}}$ de $G$ qui contient $S$ et qui est défini sur $K$.
\end{lem}

\begin{proof}
On peut se ramener aussitôt au cas où $G=Z$, dont les $K^{\text{nr}}$-tores déployés maximaux contiennent $S$ automatiquement. Observons que les orbites de l'opération de $\Gamma$ sur $\ih(Z, K^{\text{nr}})$ sont finies, car chaque appartement est associé à un tore déployé sur une extension finie de $K$. Par suite, $\Gamma$ fixe un point $x \in \ih(Z, K^{\text{nr}})$, cf. prop. 3.2.4. de \cite{BTI} et le $\oh^{\text{nr}}$-modèle parahorique $\Zh:=\Zh_x$ se descend en un $\oh$-groupe. Soient $\mathfrak{s}^{\text{nr}}$ un $\kappa$-tore maximal de $\Zh_s$ et $\sh^{\text{nr}}$ l'unique $\oh$-tore à isomorphisme unique près qui relève $\mathfrak{s}^{\text{nr}}$, en vertu de l'équivalence entre les catégories des algèbres étales et finies au-dessus de $\oh$ et $\kappa$. Grâce à la representabilité par un schéma lisse du foncteur des homomorphismes d'un tore dans un schéma en groupes lisse et affine, cf. cor. 4.2., exp. XI de \cite{SGA3}, le $\oh$-tore $\sh^{\text{nr}}$ s'envoie dans $\Zh$ par une application, qui est nécessairement une immersion fermée d'après le th. 6.8, exp. IX de \cite{SGA3}. Il s'ensuit que la fibre générique $S^{\text{nr}}:=\sh^{\text{nr}}_{\eta}$ s'identifie à un $K$-tore de $Z$, dont le $K^{\text{nr}}$-tore déduit par changement de base $K \rightarrow K^{\text{nr}}$ devient déployé maximal, car les rangs déployés de $\Zh_{\bar{s}}$ et $Z\otimes K^{\text{nr}}$ coïncident, cf. prop. 4.6.4 (i) de \cite{BTI}. 
\end{proof}

Fixons alors un tel tore $S^{\text{nr}}$, dont le appartement $\ah^{\text{nr}}:=\ah(G, S^{\text{nr}}, K^{\text{nr}})$ est invariant par $\Gamma$ et possède un point fixe. En fait $\ah:=\ah^{\text{nr}, \Gamma}$ en est un sous-espace affine de dimension égale au rang de $S^{\text{dér}}$. On note aussi $\ih^{\text{nr}}$ l'immeuble $\ih(G, K^{\text{nr}})$ de $G$ et $\ih:=\ih^{\text{nr}, \Gamma}$ la partie des points fixes par $\Gamma$. Choisissons une valuation $\varphi \in \ah$ invariante par $\Gamma$ de la donnée radicielle abstraite $(Z^{\text{nr}}(K^{\text{nr}}), U_{a^{\text{nr}}}(K^{\text{nr}}))_{a^{\text{nr}}\in \Phi^{\text{nr}}}$, où l'on note $Z^{\text{nr}}$ le centralisateur de $S^{\text{nr}}$. Pour chaque $a \in \Phi$ et $k \in \RR \cup \{+\infty\}$, posons $$U^{\text{nr}}_{a,k}=\prod_{a^{\text{nr}}}U^{\text{nr}}_{a^{\text{nr}},k} \prod_{(2a)^{\text{nr}}\in \Phi_{\text{nd}}} U^{\text{nr}}_{(2a)^{\text{nr}}, 2k}, $$ $$U_{a,k}=U^{\text{nr}}_{a,k} \cap U_a(K),$$ où $a^{\text{nr}}$ parcourt l'ensemble des racines de $G$ par rapport à $S^{\text{nr}}$ dont la restriction à $S$ est égale à $a$. Enfin on définit la fonction $\varphi_a(u)=\text{sup}\{k : u \in U_{a,k}\}$.

\begin{thm}
Les fonctions $\varphi_a$ définissent une valuation $\varphi$ de la donnée radicielle abstraite $(Z(K), U_a(K))_{a \in \Phi}$. L'immeuble $\ih(\varphi)$ associé à $\varphi$ s'identifie naturellement à $\ih$, de telle façon que l'appartement $\ah(\varphi)$ s'envoie sur $\ah$. Les racines affines de $\ih$ sont les intersections des racines affines de $\ih^{\emph{nr}}$ associées aux racines $a^{\emph{nr}} \in \Phi^{\emph{nr}} \setminus \Phi^{\emph{nr}}_0$.
\end{thm}

\begin{proof}
Pour démontrer le théorème, il faut encore vérifier les hypothèses (DI 3) et (DV 2) du th. 9.2.10. de \cite{BTI}, les restantes étant évidentes ou déjà vues. Prenons alors une facette maximale $\textbf{f}$ de $\ah^{\text{nr}}$ rencontrant $\ah$ et soit $\textbf{f}'$ une autre facette dont l'adhérence contient $\textbf{f}$ et qui rencontre $\ih$. Grâce à la bijection du th. \ref{th. paraboliques et parahoriques}, on se ramène à montrer que $\gh^0_{\textbf{f}, s}$, entendu en tant que $\kappa$-groupe, n'admet aucun sous-groupe $\kappa$-parabolique propre. Quitte à conjuguer, un tel $Q_{\textbf{f}', \textbf{f}}$ contient $\sh_s$ et donc le centralisateur $\Zh_s\supseteq \sh^{\text{nr}}_s$ ; il en résulte que $\textbf{f}'$ appartient à $\ah^{\text{nr}}$ et est invariante par $\Gamma$, ce qui entraîne $\textbf{f}'=\textbf{f}$ par hypothèse.

Quant à (DV 2), on montre notamment que l'inclusion $\Gamma_{a}'\subseteq \Gamma_{a^{\text{nr}}}'$ est une égalité. On fait usage d'abord de la structure des modèles $\uh^{\text{nr}}_{a^{\text{nr}}}$ pour qu'on trouve un élément $v \in U^{\text{nr}}_{a^{\text{nr}},k} \cap G(K_{a^{\text{nr}}}) \setminus U_{a^{\text{nr}},>k}U_{(2a)^{\text{nr}}, 2k}$, où $k \in \Gamma_{a^{\text{nr}}}'$ et $K_{a^{\text{nr}}}$ est le corps fixe par le stabilisateur $\Gamma_{a^{\text{nr}}}$ de $a^{\text{nr}}$. En écrivant $x=\prod_{\sigma \in \Gamma/\Gamma_a}\sigma(v)$ pour un ordre quelconque, on observe que $\sigma \mapsto x^{-1}\sigma(x)$ est un cocycle à valeurs dans le groupe des commutateurs des $U^{\text{nr}}_{\sigma(a^{\text{nr}}), k}$, qui est contenu dans le sous-groupe engendré par les $ U^{\text{nr}}_{\sigma(a^{\text{nr}})+\tau(a^{\text{nr}}), 2k}$. Vu le 4.6.2. de \cite{BTII}, ce sous-groupe est sous-jacent à un $\oh$-modèle de $\prod_{\sigma, \tau} U^{\text{nr}}_{\sigma(a^{\text{nr}})+\tau(a^{\text{nr}})}$, qui est un groupe vectoriel, et on lui applique le lem. 5.1.17. de \cite{BTII} afin de déduire l'annulation du premier groupe de cohomologie. En particulier, on tire $x^{-1}\sigma(x)=y^{-1}\sigma(y)$ avec $y \in \prod_{\sigma, \tau} U^{\text{nr}}_{\sigma(a^{\text{nr}})+\tau(a^{\text{nr}}), 2k}$, donc $u=yx^{-1} \in G(K)$ et, par construction même de $u$, l'on a $k \in \Gamma_a'$.

À ce stade, il s'ensuit du th. 9.2.10. de \cite{BTI} que $\varphi$ est une valuation de la donnée radicielle abstraite dont il s'agit, et les props. 9.1.17 et 9.2.15. de \cite{BTI} en fournissent une injection $\ih(\varphi)\hookrightarrow \ih$ équivariante, identifiant le membre de gauche à la plus grande partie de $\ih$ contenant $\ah$ et stable sous l'action de $G(K)$. Enfin il suffit de montrer que $\ih=G(K)\ah$. En raisonnant comme dans le lem. \ref{tore breve-deploye defini sur corps base}, on réalise un point donné $x\in \ih$ dans un appartement $\ah^{\text{nr}}_1$ invariant par $\Gamma$. Soient $S^{\text{nr}}_1$ le $K$-tore $K^{\text{nr}}$-déployé maximal associé, $S_1$ le $K$-sous-tore déployé maximal de $S^{\text{nr}}_1$ et $\ah_1=\ah^{\text{nr}}_1\cap \ih$. Vu que les tores déployés maximaux sont conjugués, on peut supposer $S_1 =S$. Grâce à la prop. 7.6.4. de \cite{BTI} (comparer avec le 5.1.3. de \cite{BTII}), on peut trouver un appartement $\ah^{\text{nr}}_2$ contenant les $S(K)$-orbites de $\varphi$ et $x$ et il s'ensuit que $\ah_1=\ah$, cf. lem. 9.2.6. de \cite{BTI}.
\end{proof}

En conclusion, nous avons construit l'immeuble $\ih(G,K)$ du groupe quasi-réductif $G$ sur le corps valué $K$ et donc on peut définir les modèles immobiliers $\gh^0_{\Omega}$ d'une partie $\Omega$ non vide et bornée d'un appartement $\ah(G, S,K)$ comme étant le $\oh$-descendu du $\oh^{\text{nr}}$-modèle correspondant $\gh^{\text{nr},0}_{\Omega}$ (ce qui a été fait mille fois pendant la démonstration !). Ceci est le plus grand modèle en groupes affine, lisse et connexe de $G$ tel que ses points avec valeurs dans $\oh^\text{nr}$ fixent $\Omega$. Il peut de même s'écrire comme recollement fermé du seul $\oh$-modèle en groupes parahorique $\Zh$ de $Z$ (qui n'est pas toujours un modèle de Néron connexe !) et des modèles en groupes convenables $\uh_{a,\Omega}$ de $U_a$, pour chaque racine $a \in \Phi_{\text{nd}}$, voir le 5.2.4. de \cite{BTII}.

La plupart des énoncés des numéros précédents se descendent de manière évident dans la topologie étale - c'est le cas de la prop. \ref{compatibilite avec solleveld} et du th. \ref{th. paraboliques et parahoriques}, dont la vérification est laissée au soin du lecteur.

\section{Applications aux grassmanniennes affines}\label{section grassmanniennes}

Ici on explique le lien étroit entre le caractère quasi-réductif d'un groupe sur un corps local ou global et la propreté de sa grassmannienne affine.
\subsection{} \label{critere geometrique de parahoricite}

Dans \cite{Rich}, T. Richarz a donné un critère géométrique en termes des grassmanniennes affines suffisant et nécessaire pour qu'un $\oh$-modèle affine, lisse et connexe d'un groupe réductif sur $K$ soit parahorique, où $K$ est un corps discrètement valué complet de caractéristiques égales à corps résiduel parfait $\kappa$. Ici nous montrons que ce critère s'étend aux groupes affines et lisses sur les corps discrètement valués complets de caractéristiques égales et résiduellement parfait. Mais d'abord on fait des rappels sur la théorie des grassmanniennes affines, ainsi qu'elles avaient été introduites de façon homogène par X. Zhu \cite{ZhuMix} :

\begin{defn}
Soient $K$ un corps discrètement valué complet, $\oh$ son anneau d'entiers et $\kappa$ le corps résiduel, qu'on suppose parfait. Étant donné un $\oh$-groupe plat et affine $\gh$, la grassmannienne affine $\text{Gr}_{\gh}$ est le pré-faisceau en $\kappa$-algèbres parfaites $R$ classifiant les classes d'isomorphisme de $\gh$-torseurs $\eh$ au-dessus de $W_{\oh}(R) $ munis d'une trivialisation $\beta$ au-dessus de $W_{\oh}(R)\otimes_{\oh} K$, où $W_{\oh}(R):=\oh\hat{\otimes}_{\kappa}R$ si $K$ est de caractéristiques égales ou $W_{\oh}(R) :=\oh\hat{\otimes}_{W(\kappa)}W(R)$ sinon.
\end{defn}

On sait que les grassmanniennes affines sont représentables par des ind-schémas ind-quasi-projectifs parfaits (voir \cite{ZhuIntro} en caractéristiques égales, auquel cas elles possèdent une réalisation canonique en tant que schémas, et \cite{BhSchGr} en inégales). Lorsque $\gh$ est lisse, la grassmannienne affine $\gh$ s'écrit comme le faisceau quotient $LG/L^+\gh$ pour la topologie étale, où l'on note $G$ la fibre générique de $\gh$, $LG$ le groupe de lacets donné par $R \mapsto G(W_{\oh}(R)\otimes K)$ et $L^+\gh$ le groupe d'arcs défini par $R \mapsto \gh(W_{\oh}(R))$. L'énoncé suivant caractérise les groupes (presque) parahoriques en fonction de la géométrie de $\text{Gr}_{\gh}$.
 
\begin{thm}\label{projectivite de la grassmannienne au cas quasi-reductif}
On garde les notations ci-dessus. Pour que $\emph{Gr}_{\gh}$ soit ind-projective, il faut et il suffit que $\gh^0$ soit un groupe parahorique.
\end{thm}

\begin{proof}
Notons qu'on peut toujours supposer $\kappa$ algébriquement clos, puisque les propriétés d'être quasi-réductif et parahorique se descendent en la topologie étale, et aussi que $\gh=\gh^0$, puisque l'application naturelle $\text{Gr}_{\gh^0} \rightarrow \text{Gr}_{\gh}$ est un recouvrement fini étale de son image fermée, dont un nombre fini de translatés recouvrent $\text{Gr}_{\gh}$. 

Traitons premièrement la suffisance. Comme $\kappa$ est algébriquement clos, $G$ est automatiquement quasi-déployé, cf. cor. \ref{quasi-deploye}. Soit $\textbf{f}$ une facette de l'appartement $\ah(G, S, K)$ de l'immeuble $\ih(G, K)$. Il sera plus convenable de considérer les fixateurs intermédiaires $\gh^1_{\textbf{f}}$, pour qu'on aie des BN-paires en niveau iwahorique. Soient $S_{w, \textbf{f}}$ les images fermées de $L^+\gh^1_{\textbf{f}} \rightarrow \text{Gr}_{\gh^1_\textbf{f}}, g\mapsto gn_w$, où $n_w$ désigne un représentant quelconque de $w \in W_{\textbf{f}}\setminus W^1_{\text{af}}/W_{\textbf{f}}$ et signalons que leur réunion forme un recouvrement fermé topologique de la grassmannienne, donc il suffit de montrer leur projectivité. En observant que $L^+\gh^1_{\textbf{f}}/L^+\gh^1_{\textbf{a}}$ est représentable par un schéma parfait lisse et projectif, on se ramène au cas où $\textbf{f}=\textbf{a}$. Après une translation par un élément de $N(K)$ qui laisse $\textbf{a}$ invariante, on peut supposer $w\in W_{\text{af}}$, ce qu'on fera désormais, et on considère une décomposition réduite $\wf=[s_{i_1}: \dots:s_{i_n}]$ de $w=s_{i_1} \dots s_{i_n}$ en des réflexions $s_i$ par rapport aux cloisons $\textbf{f}_i$ de $\textbf{a}$. Notons $D_{\wf}=L^+\gh^1_{\textbf{f}_{i_1}} \times^{L^+\gh^1_{\textbf{a}}} \dots \times^{L^+\gh^1_{\textbf{a}}} L^+\gh^1_{\textbf{f}_{i_n}}/L^+\gh^1_{\textbf{a}}$ la variété de Demazure associée à la décomposition réduite $\wf$. On montre par induction sur $n$ que ce faisceau est projectif sur $\kappa$ et le morphisme produit $D_{\wf} \rightarrow \text{Gr}_{\gh_{\textbf{a}}}$ se factorise à travers son image fermée $S_{w, \textbf{a}}$, qui est donc projective aussi.

Pour démontrer la réciproque, il faut observer d'abord que $G$ est quasi-réductif. Notons premièrement que, si $H_1 $ est un sous-$K$-groupe fermé et distingué de $H_2$, alors l'adhérence schématique $\hh_1$ de $H_1$ dans n'importe quel modèle entier plat $\hh_2$ de $H_2$ l'est aussi, et le faisceau quotient $\hh_2/\hh_1$ est représentable par un schéma en groupes affine et plat, d'après le th. 4.C. de \cite{Anan}. Par suite, l'application $\text{Gr}_{\hh_1} \rightarrow \text{Gr}_{\hh_2}$ est une immersion fermée, cf. prop. 1.2.6. de \cite{ZhuIntro}. Appliquant cela au radical unipotent déployé $R_{ud}G$ de $G$ et à une suite de composition de ce dernier à quotients successifs isomorphes à $\GG_a$, on arrive à un modèle entier plat $\ah$ de $\GG_a$ dont la grassmannienne affine est ind-projective. Notons qu'on peut trouver des $\oh$-modèles lisses et connexes $\ah_i$, $i=1,2$, de $\GG_a$ isomorphes à $\GG_a$ en tant que groupes (mais non en tant que modèles), tels qu'il y ait des morphismes $\ah_1 \rightarrow \ah \rightarrow \ah_2$. Il s'avère par des calculs explicites avec les groupes de lacets que $\text{Gr}_{\GG_a}$ s'identifie à l'espace affine infini et que le morphisme composé $\text{Gr}_{\ah_1}\rightarrow \text{Gr}_{\ah_2}$ est surjectif, d'où la même affirmation pour $\text{Gr}_{\ah}\rightarrow \text{Gr}_{\ah_2}$, ce qui est en contradiction avec l'ind-projectivité du premier. 

La dernière étape se fait comme dans la source originale et l'on suppose $\gh=\gh^0$ connexe. Comme $\gh(\oh) \subseteq G(K)$ est un sous-groupe borné pour la topologie induite par la valuation de $K$, il s'ensuit qu'il stabilise une facette de $\ih(G, K)$ (voir le th. 3.3.1. de \cite{BTI}), qu'on peut supposer maximale sans perdre de généralité, d'où une inclusion de $\gh(\oh)$ dans $P^\dagger_{\textbf{f}}$ (voir la démonstration de la prop. \ref{parahoriques a la kottwitz}). Alors le 1.7 de \cite{BTII} en fournit un morphisme schématique $\gh \rightarrow \gh_{\textbf{f}}^\dagger$ qui se factorise par connexité à travers $\gh^0_{\textbf{f}}$. La fibre sur la section unité de l'application $\text{Gr}_{\gh} \rightarrow \text{Gr}_{\gh^0_{\textbf{f}}}$ est représentable par un schéma parfait projectif et s'identifie à $ L^+\gh^0_{\textbf{f}}/L^+\gh$. Vu que le noyau de $L^+\gh^0_{\textbf{f}} \rightarrow \gh^{0, \text{réd}}_{\textbf{f}, s}$ est un $\kappa$-groupe pro-unipotent, on a nécessairement que le sous-groupe $L^+\gh$ de $L^+\gh^0_{\textbf{f}}$ est l'image réciproque d'un certain $\kappa$-sous-groupe parabolique $P$ de $\gh_{\textbf{f},s}^{0,\text{réd}}$. Par maximalité de $\textbf{f}$, on en tire $P=\gh^{0,\text{réd}}_{\textbf{f},s}$ (cf. th. \ref{parahoriques et paraboliques}), ce qui entraîne $L^+\gh=L^+\gh^0_{\textbf{f}}$, et donc l'égalité cherchée. 
\end{proof}

Observons que, lorsqu'on se donne un groupe pseudo-réductif exotique basique $G$, alors la grassmannienne affine parfaite $\text{Gr}_{\gh_{\Omega}}$ s'identifie à $\text{Gr}_{\overline{\gh}_{\Omega}}$, où l'on note $\overline{\gh}_{\Omega}$ le modèle immobilier correspondant du cousin déployé $\overline{G}$, tandis que l'application de leurs réalisations schématiques canoniques induit toujours un homéomorphisme universel entre leurs variétés de Schubert, qui n'est presque jamais birationnel. Dans \cite[§7.2]{HLR}, on trouvera une application de la théorie de Bruhat-Tits pour les groupes pseudo-réductifs au problème de déterminer tous les groupes $G$ modérément ramifiés dont les variétés de drapeaux affines sont réduites (dans le cas où le revêtement simplement connexe n'est pas étale).

Donnons-nous une courbe lisse et connexe $X$ sur un corps parfait $\kappa$. On dispose d'un ind-schéma $\text{Gr}_{\gh,X}$ au-dessus de $X$ pour chaque $X$-groupe $\gh$, qui redonne la grassmannienne affine classique au-dessus de chaque point fermé. T. Richarz a montré dans \cite{RichErr} que, pour que la fibre générique $\text{Gr}_{\gh,X,\eta}$ soit ind-projective, il faut que le $\kappa(X)$-groupe $G:=\gh^0_{\eta}$ soit quasi-réductif et il a conjecturé que ce critère est suffisant. Voici la confirmation pour les groupes pseudo-réductifs.

\begin{cor}
Soit $G$ un $\kappa(X)$-groupe pseudo-réductif. Alors, $\emph{Gr}_{G, \eta}:=\emph{Gr}_{\gh,X,\eta}$ est ind-projectif.
\end{cor}

\begin{proof}
Vu que l'ensemble des fibres propres est constructif, le fait que $\text{Gr}_{G,\eta}$ soit projectif revient à dire que $G \otimes \widehat{\oh}_{X,x}$ est parahorique presque partout, c'est-à-dire au-dessus des points fermés d'un ouvert assez petit $ U $ non vide de $X$ où $G$ est de plus encore défini. Il suffit de vérifier l'assertion pour le revêtement modérément universel $\widetilde{G}$ et pour le Cartan $Z$. 

Vu que la condition d'être parahorique se vérifie à l'aide des points entiers, on peut aussi supposer $G$ primitif. Si $G$ est simplement connexe, cela est immédiat. Sinon on considère l'épimorphisme $f:G \rightarrow \overline{G}$ défini au-dessus d'un ouvert non-vide de $X$ et induisant une égalité des points rationnels sur chaque corps local complet $\widehat{K}_{X,x}:=\widehat{\oh}_{X,x}$. On affirme que $G(\widehat{\oh}_{X,x})\rightarrow \overline{G}(\widehat{\oh}_{X,x})$ est aussi bijectif presque partout, lequel se vérifie plus commodément avec chaque terme $Z$ ou $U_a$ de la grosse cellule grâce à l'injectivité presque partout de $H^1(N, \widehat{\oh}_{X,x})\rightarrow H^1(N, \widehat{K}_{X,x})$, où $N$ est le noyau de la restriction de $f$ à $Z$ ou $U_a$. En effet, vu que $N$ s'identifie dans chacun des cas à un produit de restrictions des scalaires de groupes plats et finis, cela résulte d'une application de la suite spectrale de Leray et du critère valuatif de propreté (comparer avec la prop. 4.2.1., le lem. 4.2.5. et la prop. 4.2.9. de \cite{Ros}).

Lorsque $G=Z$ est commutatif, il faut montrer que $G$ s'étend en son propre modèle de Néron connexe au-dessus d'un ouvert assez petit. Si $G$ est de type minimal, c'est-à-dire si l'application naturelle $i_G:G \rightarrow \Res_{K'/K}G'$ vers la restriction des scalaires du quotient réductif géométrique est injective, alors l'assertion devient évidente, puisque $G'$ est un tore. Sinon on considère l'extension $1\rightarrow N \rightarrow G \rightarrow G^{\text{prtm}}\rightarrow 1$, régardée ici comme étant défini sur l'ouvert $U$, et l'on affirme que, quitte à diminuer celui-ci, alors l'application $N(\widehat{\oh}_{X,x})\rightarrow N(\widehat{K}_{X,x}) $ est bijective et $H^1(N, \widehat{\oh}_{X,x})\rightarrow H^1(N, \widehat{K}_{X,x})$ est injective pour tout $x \in U$. L'injectivité générique du premier groupe de cohomologie est un fait général déjà observé dans le paragraphe précédent. D'autre part, vu que par hypothèse $N$ n'admet qu'un nombre fini de points à valeurs dans une extension séparable, son plus grand sous-groupe fermé lisse $N^{\text{lis}}$ est fini étale, cf. \cite{CGP}, lem. C.4.1. Cet état des choses peut être supposé valide sur notre ouvert $U$, d'où l'égalité cherchée, d'après le critère valuatif de propreté appliqué à $N^{\text{lis}}$. Grace à ces hypothèses, nous vérifions sans peine que $G(\widehat{\oh}_{X,x})$ est le produit fibré de $G^{\text{prtm}}(\widehat{\oh}_{X,x})$ et $G(\widehat{K}_{X,x})$ sur $G^{\text{prtm}}(\widehat{K}_{X,x})$, lequel n'est autre que le plus grand sous-groupe borné de $G(\widehat{K}_{X,x})$, d'après la prop. \ref{modele de neron} (en effet, les caractères rationnels de $G$ et de $G^{\text{prtm}}$ coïncident).
\end{proof}

La démonstration précédente montre que le critère géométrique conjecturé par T. Richarz est étroitement liée avec la conj. I de \cite{BLR}, 10.3. De plus, si tout groupe unipotent ployé commutatif était le quotient unipotent d'un certain groupe pseudo-réductif commutatif, comme l'a demandé B. Totaro dans \cite{Tot} ques. 9.11, nous pourrions déduire de cela la vérité de cette conjecture dans le cadre ci-dessus d'une courbe lisse sur un corps parfait. En matière de cette question, quelques progrès ont été accomplis par R. Achet (voir \cite{Ach19a} et \cite{Ach19b}).

{\scriptsize \medskip\subsection*{Note ajoutée aux épreuves} Après l'écriture de ce texte, je me suis aperçu qu'il était loisible d'effacer l'hypothèse du corps résiduel parfait, tout en conservant l'excellence et l'henselité, grâce au raisonnement suivant. Étant donnés un groupe quasi-réductif $G$ quasi-déployé sur le corps $K$ et un sous-tore déployé maximal $S$ de $G$, alors l'image $S'$ de $S$ dans le quotient réductif $G'$ de $G$ sur le corps de définition $K'$ du radical unipotent géométrique est encore un sous-tore déployé maximal, faisant de $G'$ un groupe quasi-déployé au-dessus de $K'$. De plus, on obtient la relation suivante entre leurs systèmes de racines suivant $S$ resp. $S'$: $\Phi'$ est contenu dans $\Phi$ et contient toutes les racines non multipliables $\Phi_{\text{nm}}$ de $\Phi$. D'autre part, il s'avère que l'application canonique $i_G: G \rightarrow \Res_{K'/K} G'$ induit une injection entre les point rationnels des sous-groupes radiciels. On peut alors construire une valuation $\varphi$ de $G(K)$ en prenant l'image réciproque d'une valuation $\varphi'$ de $G(K)$, autrement dit en posant $\varphi_a=\varphi'_a$ ou $\frac{1}{2}\varphi'_{2a}$, selon que $a$ appartient à $\Phi'$ ou non. Ceci recadre la preuve de la prop. \ref{verificacao da valorizacao} et les restantes propositions de cet article en résultent, sans même comprendre à fond la classification de \cite{CP} (voir \cite{LouDiss} pour plus de détails).}


\begin{thebibliography}{Ana730}
%\bibitem[II]{LouGr} \textit{Grassmanniennes affines tordues sur les entiers}.\\
%\bibitem[III]{HPLR} \textit{On the normality of Schubert varieties: remaining cases in positive characteristic.}\\
%\bibitem[IV]{LouDiss} \textit{Zwei Ansätze zu lokalen Modellen}.\\
\bibitem[Ach19a]{Ach19a} R. Achet. \textit{Picard group of unipotent groups, restricted Picard functor.} 2019. hal-02063586v2.
\bibitem[Ach19b]{Ach19b} R. Achet. \textit{Unirational algebraic groups.} 2019. hal-02358528.
\bibitem[Ana73]{Anan} S. Anantharaman. \textit{Schémas en groupes, espaces homogènes et espaces algébriques sur une base de dimension 1}. Mémoires de la S.M.F., tome 33 (1973), p. 5-79.
\bibitem[BD96]{BDr} A. Beilinson \& V. Drinfeld. \textit{Quantization of Hitchin's integrable system and Hecke eigensheaves}. \url{http://math.uchicago.edu/~drinfeld/langlands/hitchin/BD-hitchin.pdf}.
\bibitem[BS17]{BhSchGr} B. Bhatt \& P. Scholze. \textit{Projectivity of the Witt vector affine Grassmannian}. Invent. Math. 209 (2017), no. 2, 329-423.
\bibitem[BoT78]{BoT2} A. Borel \& J. Tits. \textit{Théorèmes de structure et de conjugaison pour les groupes algébriques linéaires}. C. R. Acad. Sci., tome 287 (1978), 55-57.
\bibitem[BLR90]{BLR} S. Bosch, W. Lütkebohmert \& M. Raynaud. \textit{Néron Models}. Ergebnisse der Mathematik und ihrer Grenzgebiete (3), 21, Berlin, New York: Springer Verlag (1990).
\bibitem[BT72]{BTI} F. Bruhat \& J. Tits. \textit{Groupes réductifs sur un corps local : I. Données radicielles valuées}. Publ. math. de l'I.H.É.S., tome 41 (1972), p. 5-251.
\bibitem[BT84]{BTII} F. Bruhat \& J. Tits. \textit{Groupes réductifs sur un corps local : II. Schémas en groupes. Existence d'une donnée radicielle valuée}. Publ. math. de l'I.H.É.S., tome 60 (1984), p. 5-184.
\bibitem[Co12]{ConFin} B. Conrad. \textit{Finiteness theorems for algebraic groups over function fields}. Compos. Math 148 (2012), 555-639.
\bibitem[CGP15]{CGP} B. Conrad, O. Gabber \& G. Prasad. \textit{Pseudo-reductive groups}. New Mathematical Monographs 26. Cambridge: Cambridge University Press (ISBN 978-1-107-08723-1/hbk; 978-1-316-09243-9/ebook). xxiv, 665 p. (2015).
\bibitem[CP16]{CP} B. Conrad \& G. Prasad. \textit{Classification of pseudo-reductive groups}. Annals of Mathematics Studies 191, Princeton University Press (2016).
\bibitem[HLR]{HLR} T. Haines, J. Lourenço \& T. Richarz. \textit{On the normality of Schubert varieties: remaining cases in positive characteristic.} En préparation.
\bibitem[HR08]{HR} T. Haines \& M. Rapoport. \textit{On parahoric subgroups}. Adv. Math. 219 (2008), 188-198.
\bibitem[La96]{LndvCpc} E. Landvogt. \textit{A compactification of the Bruhat-Tits building}. Lect. Notes in Math., vol. 1619, Springer-Verlag, Berlin, 1996.
\bibitem[Lan52]{Lan52} S. Lang. \textit{On quasi algebraic closure}. Annals of Mathematics, Second Series, Vol. 55, No. 2 (1952), 373-390.
\bibitem[Lou19]{LouGr} J. Lourenço. \textit{Grassmanniennes affines tordues sur les entiers}. \url{https://arxiv.org/abs/1912.11918}.
\bibitem[Lou20]{LouDiss} J. Lourenço. \textit{Zwei Ansätze zu lokalen Modellen}. En préparation.
\bibitem[MRR20]{RichDil} A. Mayeux, T. Richarz \& M. Romagny. \textit{Torsors under Néron blow-ups}. \url{https://arxiv.org/abs/2001.03597}.
\bibitem[Oes84]{Oes} J. Oesterlé. \textit{Nombres de Tamagawa et groupes unipotents en caractéristique p}. Inventiones mathematicae 78 (1984): p. 13-88.
\bibitem[PR08]{PR} G. Pappas \& M. Rapoport. \textit{Twisted loop groups and their affine flag varieties}. Adv. Math. 219, No. 1, 118-198 (2008).
\bibitem[PZ13]{PZh} G. Pappas \& X. Zhu. \textit{Local models of Shimura varieties and a conjecture of Kottwitz}. Invent. Math. 194, No. 1, 147-254 (2013).
\bibitem[Pr16]{PrDesc} G. Prasad. \textit{A new approach to unramified descent in Bruhat-Tits theory}. \url{https://arxiv.org/abs/1611.07430}.
\bibitem[Ray70]{Ray} M. Raynaud. \textit{Faisceaux amples sur les schémas en groupes et les espaces homogènes}. Lecture Notes in Mathematics, Vol. 119, Springer-Verlag, Berlin, 1970.
\bibitem[Ri16]{Rich} T. Richarz. \textit{Affine Grassmannians and geometric Satake equivalences.} Int. Math. Res. Not., Vol. 2016, Issue 12, 3717–3767 (2016).
\bibitem[Ri19]{RichErr} T. Richarz. \textit{Erratum to \textquotedblleft Affine Grassmannians and geometric Satake equivalences\textquotedblright}. Int. Math. Res. Not., rnz210 (2019).
\bibitem[Ros18]{Ros} Z. Rosengarten. \textit{Tate duality in positive dimensions over function fields}. \url{https://arxiv.org/pdf/1805.00522.pdf}.
\bibitem[SS97]{SchSt} P. Schneider \& U. Stuhler. \textit{Representation theory and sheaves on the Bruhat-Tits building}. Publ. math. de l'IHÉS, tome 85 (1997), 97-191.
\bibitem[SGA3]{SGA3} SGA 3. \textit{Schémas en groupes. III : Structure des schémas en groupes réductifs}. Séminaire de Géométrie Algébrique du Bois Marie 1962/64 (SGA 3). Dirigé par M. Demazure et A. Grothendieck. Lecture Notes in Mathematics, Vol. 153, Springer-Verlag, Berlin, 1962/1964.
\bibitem[Sol18]{Sol} M. Solleveld. \textit{Pseudo-reductive and quasi-reductive groups over non-archimedean local fields}. Journal of Algebra 510 (2018), 331-392.
\bibitem[St59]{St1} R. Steinberg. \textit{Variations on a theme of Chevalley}. Pacific J. Math., Vol. 9, No. 3, 875-891 (1959).
\bibitem[St68]{St2} R. Steinberg. \textit{Endomorphisms of linear algebraic groups.} Memoirs of the American Mathematical Society, No. 80 (1968).
\bibitem[Ti71]{TitsRep} J. Tits. \textit{Représentations linéaires irréductibles d'un groupe réductif sur un corps quelconque}. Journal für die reine und angewandte Mathematik 247 (1971): 196-220.
\bibitem[Ti84]{TitsOber} J. Tits. \textit{Groups and group functors attached to Kac-Moody data}. Arbeitstagung Bonn 1984, p. 193-223.
\bibitem[Ti86]{TitsImAff} J. Tits. \textit{Immeubles de type affine}. Buildings and the geometry of diagrams, Como (1984), L. A. Rosati, Lect. Notes in Math., Springer Verlag, Berlin 1986, 159-190. 
\bibitem[Ti87]{TitsUniq} J. Tits. \textit{Uniqueness and presentation of Kac-Moody groups over fields}. J. Algebra 105 (1987), nº 2, 542-573.
\bibitem[Ti89]{TitsBour} J. Tits. \textit{Groupes associés aux algèbres de Kac-Moody.} Astérisque, tome 177-178 (1989), Séminaire Bourbaki, exp. nº 700, p. 7-31.
\bibitem[Ti92]{TitsTwin} J. Tits. \textit{Twin buildings and groups of Kac-Moody type}, Groups, combinatorics and Geometry, Durham 1990, London Math. Soc. Lect. Series 165, Cambridge University Press (1992), 249-286.
\bibitem[Tot13]{Tot} B. Totaro. \textit{Pseudo-abelian varieties}. Annales scientifiques de l'École Normale Supérieure, serie 4, tome 46 (2013) no. 5, p. 693-721.
\bibitem[WW80]{WW} W. C. Waterhouse \& B. Weisfeiler. \textit{One-dimensional affine group schemes}. J. Algebra, Vol. 66 (1980), nº 2, p. 550-568.
\bibitem[Yu15]{Yu} J.-K. Yu, \textit{Smooth models associated to concave functions in Bruhat-Tits theory}. Autour des schémas en groupes, vol. III, 227-258, Panor. Synthèses, 47, Soc. Math. France, Paris (2015).
\bibitem[Zhu16]{ZhuIntro} X. Zhu. \textit{An introduction to affine Grassmannians and the geometric Satake equivalence}. \url{https://arxiv.org/abs/1603.05593}.
\bibitem[Zhu17]{ZhuMix} X. Zhu. \textit{Affine Grassmannians and the geometric Satake in mixed characteristic}. Ann. of Math. (2) 185 (2017), nº 2, 403-492.
\end{thebibliography}
\end{document}